\documentclass[envcountsect]{svjour3} 
\usepackage{srcltx}
\usepackage{xcolor}
\usepackage{graphicx}
\usepackage{amsmath}
\usepackage{amssymb}
\usepackage{theorem}
\usepackage{euscript}
\usepackage{epic,eepic}
\usepackage{pstricks}
\usepackage{enumitem}
\usepackage{parskip}
\usepackage{cleveref}
\usepackage{tikz}
\usepackage{color}
\usepackage{multirow}
\usepackage{float}
\usepackage{multirow, array}
\usepackage{booktabs}
\usepackage{mathrsfs}
\usepackage{mathdots}
\usepackage{stackrel}
\usepackage{relsize}
\usepackage{ dsfont }
\usepackage{listings}
\usepackage{color}
\usepackage{graphicx}
\usepackage{setspace} 
\usepackage[utf8]{inputenc}
\usepackage{xcolor}
\usepackage[USenglish]{babel}
\usepackage{mathtools}
\usepackage{subfigure}
\usepackage{amsmath}
\usepackage{amssymb}
\usepackage{theorem}
\usepackage{euscript}
\usepackage{epic,eepic}
\usepackage{pstricks}
\usepackage{parskip}
\usepackage{tikz}
\usepackage{multirow}
\usepackage{float}
\usepackage{array}
\usepackage{booktabs}
\usepackage{mathrsfs}
\usepackage{mathdots}
\usepackage{stackrel}
\usepackage{cite}
\usepackage{fancybox, calc}
\usepackage{verbatim}
\usepackage{fancyhdr}
\definecolor{mygreen}{rgb}{0,0.6,0}
\definecolor{mygray}{rgb}{0.5,0.5,0.5}
\definecolor{mymauve}{rgb}{0.58,0,0.82}
\definecolor{labelkey}{rgb}{0,0.08,0.45}
\definecolor{refkey}{rgb}{0,0.6,0.0}
\definecolor{Brown}{rgb}{0.45,0.0,0.05}
\definecolor{dgreen}{rgb}{0.00,0.49,0.00}
\definecolor{dblue}{rgb}{0,0.08,0.75}
\DeclareSymbolFont{largesymbolsA}{U}{jkpexa}{m}{n}
\SetSymbolFont{largesymbolsA}{bold}{U}{jkpexa}{bx}{n}
\DeclareMathSymbol{\varprod}{\mathop}{largesymbolsA}{16}

\author{Cristian Vega}
\title{\sffamily Alternating and randomized  on monotone inclusions
\footnote{Contact author: 
Cristian Vega Cere\~no, {\ttfamily cristian.vega.14@sansano.usm.cl},
}}
\author{Cristian Vega$^1$ 
\\[5mm]
\small $\!^1$Universidad T\'ecnica Federico Santa Mar\'ia\\
\small Departamento de Matem\'atica\\
}

\date{\ttfamily \today}
\newcommand{\Frac}[2]{\displaystyle{\frac{#1}{#2}}} 

\newcommand{\scal}[2]{{\left\langle{{#1}\mid{#2}}\right\rangle}}
\newcommand{\pscal}[2]{\langle\langle{#1}\mid{#2}\rangle\rangle} 
 
\newcommand{\menge}[2]{\big\{{#1}~ \colon~{#2}\big\}}

\newcommand{\HH}{\ensuremath{{\mathcal H}}}
\newcommand{\HHH}{\ensuremath{\boldsymbol{\mathcal H}}}

\newcommand{\GG}{\ensuremath{{\mathcal G}}}

\newcommand{\emp}{\ensuremath{{\varnothing}}}

\newcommand{\Id}{\ensuremath{\operatorname{Id}}\,}
\newcommand{\RR}{\ensuremath{\mathbb{R}}}

\newcommand{\ran}{\ensuremath{\operatorname{ran}}}

\newcommand{\RP}{\ensuremath{\left[0,+\infty\right[}}

\newcommand{\RPP}{\ensuremath{\left]0,+\infty\right[}}

\newcommand{\RX}{\ensuremath{\left]-\infty,+\infty\right]}}

\newcommand{\NN}{\ensuremath{\mathbb N}}

\newcommand{\exi}{\ensuremath{\exists\,}}

\newcommand{\weakly}{\ensuremath{\:\rightharpoonup\:}}
\newcommand{\argmin}{\ensuremath{\operatorname{argmin}}}

\newcommand{\Fix}{\ensuremath{\operatorname{Fix}}}
\newcommand{\dom}{\ensuremath{\operatorname{dom}}}
\newcommand{\prox}{\ensuremath{\operatorname{prox}}}

\newcommand{\sri}{\ensuremath{\operatorname{sri}}}

\newcommand{\infconv}{\ensuremath{\mbox{\small$\,\square\,$}}}

\newcommand{\vertiii}[1]{{\left\vert\kern-0.25ex\left\vert\kern-0.25ex\left\vert #1 
\right\vert\kern-0.25ex\right\vert\kern-0.25ex\right\vert}}
\newcommand{\normi}{\vert\kern-0.25ex\vert\kern-0.25ex\vert}

\renewcommand{\abstractname}{Abstract}

\ProvideTextCommand{\DJ}{OT1}{\raisebox{0.25ex}{-}\kern-0.4em D}
\renewcommand\theenumi{(\roman{enumi})}
\renewcommand\theenumii{(\alph{enumii})}

\numberwithin{equation}{section}
\smartqed 
\usepackage{graphicx}

\begin{document}

\title{Random Activations in Primal-Dual Splittings for 
Monotone Inclusions with a priori Information}

\author{Luis Briceño-Arias   \and  Julio Deride \and Cristian Vega}

\institute{ Luis Briceño-Arias \at
             Universidad Técnica 
Federico Santa María \\
              Santiago, Chile, luis.briceno@usm.cl  
           \and
           Julio Deride \at
             Universidad Técnica 
Federico Santa María \\
              Santiago, Chile, julio.deride@usm.cl
              \and 
              Cristian Vega \at
             Universidad Técnica 
Federico Santa María \\
              Santiago, Chile, cristian.vega.14@sansano.usm.cl
}

\date{Received: date / Accepted: date}

\maketitle

\begin{abstract}
In this paper, we propose a numerical approach for solving 
composite primal-dual monotone inclusions with a priori information. 
The 
underlying a priori information set is represented by the intersection of 
fixed point sets 
of a finite number of operators,  and we propose and algorithm that 
activates the corresponding set 
by following a finite-valued random variable at each iteration. 
Our formulation is flexible and includes, for instance,
deterministic and Bernoulli 
activations over cyclic schemes, and Kaczmarz-type random 
activations. The 
almost sure convergence of the algorithm 
is obtained by means of properties of stochastic Quasi-Fej\'er 
sequences. We also recover several 
primal-dual algorithms for monotone inclusions in the 
context without a priori information and classical algorithms for 
solving convex feasibility problems and linear systems. 
In the context of convex optimization with inequality constraints, 
any selection of the constraints defines the a priori
information set, in which case the operators involved are simply 
projections onto half spaces. By incorporating random projections 
onto a selection of the constraints to classical primal-dual schemes, 
we obtain faster algorithms as we illustrate by means of a numerical 
application to a stochastic
arc capacity expansion 
problem in a transport network.
\end{abstract}


\section{Introduction}

We devote this paper to develop an efficient numerical algorithm for 
solving primal–dual monotone inclusions 
involving a priori information on the  primal solutions. Primal-dual 
inclusions have found many applications, such as evolution inclusions 
\cite{Attouch-BA, Peypouquet-Sorin}, 
variational inequalities\cite{Siopt1,Gabay,Tseng I}, partial differential 
equations (PDEs) \cite{ABRS, Gabay}, and Nash 
equilibria \cite{Nash equilibria}. In particular, when the monotone 
operators are 
subdifferentials of convex functions, the inclusion we study reduces to 
an optimization problem with a priori information. This problem arises 
in many applications, such as arc capacity expansion in traffic networks
\cite{Xiaojun Chen}, image recovery \cite{BACPP,Peypouquet, 
	Chambolle-Lions}, and signal processing \cite{prox,mms1,SVVA1}.

The a priori information is modeled by  the  intersection of fixed point 
sets of a  finite  number of averaged nonexpansive operators. In the 
particular case when these operators are projections onto closed 
convex sets, the a priori information is represented by their 
intersection. 
More generally, if the operators are resolvents,
the a priori information set models common solutions to several
convex optimization problems and monotone inclusions. 

In the absence of a priori information, the studied problem can be 
solved by \cite{Vu}, and by \cite{Condat} in the convex optimization 
case.
These methods generalize several classical algorithms for monotone 
inclusions and convex optimization as, e.g., 
\cite{Lionsmercier,PPA,Chambolle- Pock}. 
In the case when 
the a priori 
information is the fixed point set of a single operator,
an extension of the method in \cite{Vu} including 
the activation of the operator at each iteration is proposed
in \cite{BA-Lopez}. This 
method is applied to linearly constrained convex optimization 
problems, 
in which the a priori information set is a selection of the constraints.
This formulation leads to an extension of the method in \cite{Condat}, 
which includes a projection onto the set defined by the selected 
constraints. This enforces 
feasibility on primal iterates on the chosen set, resulting in a more 
efficient algorithm as 
verified numerically.

The previous approach opens the question on finding
an appropriate manner to select and 
project onto the constraints, in order to induce more efficient methods. 
In the particular context of convex feasibility, several projecting 
schemes are proposed in the literature, e.g.,  
\cite{Kaczmarz,Strohmer-Vershynin,Convexfeasibility,plc97}.
In the case of solving overdetermined consistent linear systems,  a 
cyclic deterministic projection scheme over the 
hyperplanes generated by each of the equalities is proposed 
in \cite{Kaczmarz}. A randomized version of the method in 
\cite{Kaczmarz} is derived in
\cite{Strohmer-Vershynin}, where the probability of activation of each 
hyperplane is proportional to the  
the size of its normal vector.
As a consequence, the method exhibits an exponential 
convergence rate in expectation.
Beyond consistent linear systems, several projecting 
schemes are proposed in \cite{Convexfeasibility,plc97} for the convex 
feasibility problem, including static, cyclic, and quasi-cyclic projections.
A random block coordinate method using parallel 
Bernoulli activation is proposed in \cite{Siopt6} for the resolution of 
monotone inclusion problems.

In the context of convex optimization with a priori information defined 
by the intersection of convex sets, this paper aims at combining 
previous projection schemes with the primal-dual method in 
\cite{Condat}. In the more general context of monotone inclusions 
with a priori information, our goal is to extend this idea to combine 
several activation schemes on the operators 
defining the a priori information set with the primal-dual splitting in 
\cite{Vu} for monotone inclusions. As a result, we obtain a 
generalization of the 
methods in \cite{BA-Lopez,Vu,Kaczmarz,Strohmer-Vershynin,Siopt6}  
and a unified manner to activate the operators, including 
\cite{Kaczmarz,Strohmer-Vershynin,Siopt6} and some schemes in 
\cite{Convexfeasibility,plc97}.

We illustrate the numerical efficiency of our method in the 
arc capacity expansion problem in transport networks, 
corresponding to a convex optimization with linear inequality 
constraints. We provide 13 algorithms by varying the projecting 
schemes, and we compare their performance for solving this problem.
We observe an improvement up to 35\% in computational time
for the algorithms including randomized and alternating 
projections with respect to the method without projections, 
justifying the 
advantage of our approach.

This paper is organized as follows. In Section~\ref{sec:NP} we 
introduce our notation and some preliminaries. In 
Section~\ref{sec:Main} we provide the 
main algorithm, we prove its almost sure weak convergence, 
and we exploit the flexibility of our approach obtaining several 
schemes available in the literature. 
In Section~\ref{sec:traffic} we implement different activation schemes
in the context of the
arc capacity expansion problem in transport networks and we compare 
their efficiency with respect to the method without any activation.
Finally, we provide some conclusions and perspectives.

\section{Notation and Preliminaries}
\label{sec:NP}
Troughout this paper, $\HH$ stands for a real separable Hilbert space,
the identity operator on $\mathcal{H}$ is denoted by $\Id$, and 
$\rightharpoonup$ and $\rightarrow$ denote
weak and strong convergence in $\mathcal{H}$, respectively. The set 
of weak sequential cluster points of a sequence 
$(x_n)_{n\in\mathbb{N}}$ in $\mathcal{H}$ is denoted by 
$\mathfrak{W}\left(x_{n}\right)_{n \in \mathbb{N}}$. 
The projector operator onto a nonempty closed convex set $S 
\subset \mathcal{H}$ is denoted by $P_{S}$, its normal 
cone is denoted by $N_{S}$,
and its strong relative 
interior is denoted by $\sri(S)$. Given $\alpha\in\left]0, 
1\right[$, an 
operator $T :\mathcal{H} \rightarrow 
\mathcal{H}$ is $\alpha$-averaged nonexpansive iff, for  every $x$ 
and $y$ in $\HH$, we have $\|T 
x - Ty\|^2 \leq \|x - y\|^2 -\frac{1-\alpha}{\alpha}\|(\Id-T)x- (\Id-T)y\|^2$. 
Let 
$M:\mathcal{H}\rightrightarrows \mathcal{H}$ be a set-valued 
operator. We denote by $\operatorname{dom} M$ the domain of 
$M$,  by $\ran(M)$ its range of $M$, and by $\operatorname{gra} M$
its graph. The inverse $M^{-1}$ of $M$ is the operator defined by
$M^{-1}\colon u\mapsto \menge{x\in\HH}{u \in M x}$. Given 
$\rho\geq 0$, $M$ is 
$\rho-$strongly monotone iff, for every $(x,u)$ and $(y,v)$ in 
$\operatorname{gra}(M)$, $\langle x - y\mid u - v\rangle \geq \rho 
\|x - y\|^2$, it is $\rho-$cocoercive iff $M^{-1}$ is $\rho-$strongly 
monotone, $M$ is monotone iff it is $0-$strongly monotone, and it is 
maximally monotone iff its graph is maximal, in the 
sense of inclusions in $\HH \times \HH$, among the graphs of 
monotone operators. The resolvent of $M$ is denoted by 
$J_{M}=(\Id+M)^{-1}$. If $M$
is maximally monotone, then $J_{M}$ is single-valued and 
$1/2-$averaged nonexpansive operator, and $\dom J_{M} = 
\mathcal{H}$. The parallel 
sum of $A:\mathcal{H}\rightrightarrows \mathcal{H}$ and 
$B:\mathcal{H}\rightrightarrows \mathcal{H}$ is defined by $A\infconv B:=\left(A^{-1}+B^{-1}\right)^{-1}$.
We denote by $\Gamma_{0}(\mathcal{H})$ the set of proper, lower
semicontinuous and convex functions from $\HH$ to $\RX$.  The 
subdifferential of $f\in \Gamma_{0}(\mathcal{H})$, denoted by 
$\partial f$, is 
maximally monotone; if $f$ is G\^ateaux differentiable in $x$ then 
$\partial f(x)=\{\nabla f(x)\}$, and we have $(\partial f)^{-1}=\partial 
f^{*}$, where 
$f^{*}\in \Gamma_{0}(\mathcal{H})$ is the Fenchel conjugate of $f$. 
The infimal convolution of 
the two functions $f$ and $g$ from $\HH$ to $]-\infty,+\infty[$ is 
defined by $f\infconv g:x\rightarrow\inf_{y\in\HH}\left(f(y)+g(x-y)\right)$. The proximal operator of $f\in \Gamma_{0}(\mathcal{H})$ is defined by $\operatorname{prox}_{f}: x\mapsto \underset{y \in \mathcal{H}}{\operatorname{argmin}} f(y)+\frac{1}{2}\|x-y\|^{2}$ and we have $J_{\partial f}=\operatorname{prox}_{f}$. 
Moreover, if $S\subset\mathcal{H}$ is a nonempty convex closed 
subset, then $\iota_{S}\in\Gamma_{0}(\HH)$, $N_{S}=\partial 
\iota_{S}$, and $J_{N_{S}}=P_{S}$, where $\iota_{S}$ assigns to 
$x\in\HH$ the value $0$ if $x$ 
belongs to $S$ and $+\infty$, otherwise. For further results on 
monotone operator theory and convex optimization, the reader is 
referred to \cite{Livre1}.

Throughout this paper $(\Omega,\mathcal{X},\mathbb{P})$ is a fixed 
probability space. The 
space of all random variables $z$ with values in $\HH$ such that 
$\|z\|$ is integrable is denoted by 
$L^{1}(\Omega,\mathcal{X},\mathbb{P};\mathcal{H})$. Given a 
$\sigma$-algebra $\mathcal{E}$ of $\Omega$, $x\in 
L^{1}(\Omega,\mathcal{X},\mathbb{P};\mathcal{H})$, and $y\in 
L^{1}(\Omega,\mathcal{X},\mathbb{P};\mathcal{H})$, $y$ is the 
conditional expectation of $x$ with respect to $\mathcal{E}$ iff, 
for every $E\in 
\mathcal{E}$, $\int_{E}xd\mathbb{P}=\int_{E}yd\mathbb{P}$,
in which case we write $y 
=\mathbb{E}(x\hspace{0.05cm}|\hspace{0.05cm}\mathcal{E}) $. The 
characteristic function on $D\subset\Omega$ is denote by 
$\mathds{1}_D$, which is 1 in $D$ and 0 otherwise. An 
$\mathcal{H}-$valued random variable is a measurable map 
$x:(\Omega,\mathcal{X})\rightarrow (\mathcal{H},\mathcal{B})$, 
where $\mathcal{B}$ is the Borel $\sigma$-algebra. The 
$\sigma$-algebra generated by a family $\Phi$ of random variables is 
denoted by $\sigma(\Phi)$. Let $\mathscr{X}=(\mathcal{X}_n)_{n\in 
\mathbb{N}}$ be a sequence of sub-sigma algebras of $\mathcal{X}$ 
such that $(\forall n\in \mathbb{N})\quad\mathcal{X}_{n}\subset 
\mathcal{X}_{n+1}$. We denote by $\ell_+(\mathscr{X})$ the set of 
sequences of $\RP-$valued random variables $(\xi_{n})_{n\in 
\mathbb{N}}$ such that, for every $n\in \mathbb{N}$, $\xi_n$ is 
$\mathcal{X}_{n}$-measurable. We set \begin{align}
(\forall p\in\RPP ) \quad 
\ell_+^{p}(\mathscr{X}):=\left\{\left(\xi_{n}\right)_{n\in \mathbb{N}}\in 
\ell_+(\mathscr{X})\Bigg|\normalsize \sum_{n\in 
\mathbb{N}}\xi_{n}^{p}<+\infty \quad \mathbb{P}\text{-a.s.} 
\right\}.\label{equation 1.9}
\end{align}Equalities and inequalities involving random variables will 
always be understood to hold $\mathbb{P}-$almost surely, even if the 
expression “$\mathbb{P}-$a.s.” is not explicitly written.
\\ \\The following result is an especial case of 
\cite[Theorem~1]{Robb71} and is the main tool to prove the 
convergence of 
Stochastic Quasi-Fej\'er sequences.
\begin{lemma}
\label{l: lemma 1.4}
\textit{\cite[Proposition 2.3]{Siopt6}} Let $\boldsymbol{Z}$ be a 
nonempty 
closed subset of a real Hilbert space $\HHH$, let 
$\left(\boldsymbol{x}^n\right)_{n\in\NN}$ be a sequence of
$\HHH$-valued random variables, and let 
$\mathscr{X}=(\mathcal{X}_{n})_{n\in\mathbb{N}}$ be a 
sequence of sub-sigma algebras of $\mathcal{X}$ such that, for 
every $n\in 
\mathbb{N}$, $\mathcal{X}_n\subset\mathcal{X}_{n+1}$. 
Let $(b_n)_{n\in\mathbb{N}}\in\ell_+(\mathscr{X})$ 
be such that
\begin{align}
(\forall n\in\mathbb{N})\hspace{0.4cm} 
\mathbb{E}(\|\boldsymbol{x}^{n+1}-\boldsymbol{z}\|^2\mid\mathcal{X}_n)+
b_n\leq 
\|\boldsymbol{x}^{n}-\boldsymbol{z}\|^2\hspace{0.4cm} 
\mathbb{P}-\text{a.s.}
\end{align}
Then, $(b_n)_{n\in\NN}\in\ell^1_+(\mathscr{X})$. Moreover,
if $\mathfrak{W}(\boldsymbol{x}^n)_{n\in 
    \mathbb{N}}\subset 
    \boldsymbol{Z}\hspace{2mm}\mathbb{P}-$a.s., then
$\left(\boldsymbol{x}^n\right)_{n\in\NN}$ converges weakly 
    $\mathbb{P}-$a.s. to a $\boldsymbol{Z}-$valued random variable.
\end{lemma}

\section{Main Problem and Algorithm}
\label{sec:Main}
We consider the following problem.
\begin{problem}
\label{problema 1}
Let $A:\mathcal{H}\rightrightarrows \mathcal{H}$, 
$B:\mathcal{G}\rightrightarrows \mathcal{G}$,  and 
$D:\mathcal{G}\rightrightarrows \mathcal{G}$ be
maximally monotone operators such that $D$ is $\delta-$strongly 
monotone, for some $\delta>0$, let $C:\mathcal{H}\rightarrow 
\mathcal{H}$ be a 
$\mu-$cocoercive operator, for some $\mu>0$,
and let $L:\mathcal{H}\rightarrow\mathcal{G}$ be a nonzero
bounded linear operator. For 
every $i\in\{1,\ldots,m\}$, let $T_i\colon\HH\to\HH$ be a $\alpha_{i}-$ 
averaged nonexpansive operator, for some $\alpha_i\in\left]0,1\right[$, 
and set $S:=\bigcap_{i=1}^{m} 
\Fix\left(T_i\right)$. The 
problem is to find $(x,u)\in (S\times\GG)\cap Z_0$, where 
$Z_0$ is the 
set of primal-dual solutions to  
\begin{equation}
\label{e:Z0}
\text{find } (x,u)\in \HH\times \GG 
\hspace{4mm}\text{s.t.}\hspace{4mm}
\begin{cases}
0\in Ax+L^*u+ Cx,\\
0\in B^{-1}u-Lx+D^{-1}u,
\end{cases}
\end{equation}
and we assume that  
$Z:=(S\times\GG)\cap Z_0\ne\emp$.
\end{problem}
The set $S$ represents a priori information on the primal solution. In 
the case when, for every $i\in\{1,\ldots,m\}$, $T_i=P_{S_i}$ for 
nonempty closed convex sets $(S_i)_{1\le i\le m}$, the a priori 
information set is simply $S=\cap_{i=1}^mS_i$. Alternatively if,  
for every $i\in\{1,\ldots,m\}$, $T_i=J_{M_i}$ for some maximally 
monotone operator $M_i$ defined in $\HH$, 
Problem~\ref{problema 1} reduces to find a common solution to a 
finite number of monotone inclusions, and this approach works in 
more 
general settings by choosing the operators $(T_i)_{1\le i\le m}$ 
appropriately. 

In \cite{BA-Lopez}, Problem \ref{problema 1} is solved when $m=1$, 
by including a deterministic activation of $T_1$. 
This method reduces to the method in \cite{Vu} in the case when 
$T_1=\Id$ and, hence, $S=\HH$.

In the particular case when $A=\partial F$, $B=\partial G$, $C=\nabla 
H$, and $D=\partial 
\ell$, where $F\in\Gamma_{0}(\HH)$, $G\in\Gamma_{0}(\GG)$, 
$H:\HH\rightarrow\mathbb{R}$ is a differentiable convex function with 
$\mu^{-1}-$Lipschitz gradient, and $\ell\in\Gamma_{0}(\GG)$ is a 
$\delta-$ strongly convex, every solution $(x,u)\in Z$ is a solution to 
the following primal optimization problem with a priori information
\begin{align}
\tag{$\mathcal{P}_{1}$}
\mbox{find}\,\, x\in S\cap \argmin_{x\in\HH} 
(F(x)+\left(G\infconv\ell\right)(Lx)+H(x))
\label{equation 1.1}
\end{align}
and its dual problem
\begin{align}
\tag{$\mathcal{D}_1$}
\mbox{find}\,\, u\in\argmin_{u\in\GG} 
(G^{*}(u)+(F+H)^{*}(-L^{*}u)+\ell^{*}(u)).
\label{equation 1.2}
\end{align}
Moreover, if the following qualification condition is 
satisfied
\begin{align}
\label{equation 1.10}
0\in \sri\big(L\left(\dom F\right)-\left(\dom G+\dom\ell\right)\big).
\end{align}
then, by \cite[Proposition 4.3]{svva2} the sets of solutions coincide.
In \cite{Chambolle- Pock} a primal-dual algorithm for solving 
\eqref{equation 1.1}-\eqref{equation 1.2} is proposed, in the case 
when $H=0$, 
$\ell=\delta_{\{0\}}$, and $S=\HH$. In \cite{Condat} the previous 
algorithm is extended to the case $H\ne0$. From \cite{BA-Lopez} we 
obtain a method to solve the case when $m=1$ and $S=\Fix T_1$,
which incorporates a deterministic activation of $T_1$ at each iteration.

A particular instance of \eqref{equation 1.1} is the resolution of 
overdetermined consistent linear systems, in which 
\eqref{equation 1.1} reduces to find $x\in S=\cap_{i=1}^mS_i$,
where, for every $i\in\{1,\ldots,m\}$, $S_i$ is the hyperplane defined 
by a linear equation $r_i^{\top}x=b_i$ in finite dimensions.
In this context, the Kaczmarz 
method implements cyclic projections onto the hyperplanes are 
converging to a feasible 
solution to the problem in \cite{Kaczmarz}.  A randomized version of 
the Kaczmarz method is proposed in in 
\cite{Strohmer-Vershynin} for solving consistent and 
overdetermined linear systems. This Randomized Kaczmarz algorithm 
has an exponential convergence rate  in expectation. In the convex 
feasibility setting, in which $(S_i)_{1\le i\le m}$ are general close 
convex sets, alternative converging projecting schemes are 
proposed in \cite{Convexfeasibility,plc97}.

Previous projecting schemes motivates the following result, which 
combines randomized/alternating activation of 
$\{\Id,T_1,\ldots,T_m\}$ with the primal-dual method in \cite{Vu}.
Our method extends the fixed activation scheme proposed in
\cite{BA-Lopez}. We obtain a weakly convergent 
$\mathbb{P}-$a.s. algorithm, where $\mathbb{P}$ is the probability 
measure associated 
with the the sequence of random variables modelling the operator 
activation. These sequences are defined in the probability space 
$(\Omega,\mathcal{X},\mathbb{P})$.

\begin{theorem}
\label{A:algorithm 4.1}
Consider the setting of Problem~\ref{problema 1}. Let 
$\tau\in\left]0,2\mu\right[$, let $\gamma \in\left]0,2\delta\right[$ be 
such that 
\begin{align}
\| L\|^2< 
\left(\frac{1}{\gamma}-\frac{1}{2\delta}\right)\left(\frac{1}{\tau}-
\frac{1}{2\mu}\right).
\label{equation 4.1.0}
\end{align}
Let $(x^0,\overline{x}^0,u^0)\in 
\mathcal{H}\times\mathcal{H}\times\mathcal{G}$ be such that 
$x^0=\overline{x}^0$ and set $I=\{0,1,...,m\}$. Let 
$(\epsilon_k)_{ k\in \mathbb{N}}$ be a sequence of independent 
random variables such that,  for every 
$k\in\NN$, $\epsilon_k(\Omega)= I_{k}\subset I$ and 
consider the following routine
\begin{align}
(\forall k\in\NN)\hspace{0.2cm}
&\left\lfloor 
\begin{array}{ll}
u^{k+1}=J_{\gamma B^{-1}}\left(u^k+\gamma 
\left(L\bar{x}^k-D^{-1}u^{k}\right)\right)\\
p^{k+1}=J_{\tau A}(x^k-\tau( L^*u^{k+1}+Cx^k))\\
x^{k+1}=T_{\epsilon_{k+1}}p^{k+1}\\
\bar{x}^{k+1}=
x^{k+1}+ p^{k+1}-x^{k},
\label{equation 4.1}
\end{array}
\right.
\end{align}
where $T_{0}=\Id$.
Suppose that one of the following hold:
\begin{enumerate}
\item 
\label{A:algorithm 4.10}
{ $Z_0\subset S\times\GG$.}
\item 
\label{A:algorithm 4.1i}
There exists $N\in\NN\smallsetminus\{0\}$ such that $(\forall 
n\in\mathbb{N})\hspace{0.2cm}I=\bigcup_{k=n}^{n+N-1}I_{k}$ and
\begin{equation}
\label{A:algorithm 4.1ii}
0<\zeta:=\inf\limits_{k\in\mathbb{N}}\min\limits_{i\in I_k\backslash 
	\{0\}}\pi^i_{k},
\end{equation}
where, for every $k\in\NN$ and $i\in I_k$, 
$\pi_k^{i}=\mathbb{P}(\epsilon_k^{-1}(\{i\}))$.
\end{enumerate}
Then $((x^k,u^k))_{k\in\mathbb{N}}$ converges weakly 
$\mathbb{P}-$a.s. to a 
$Z-$valued random variable.
\end{theorem}
\begin{proof}
Fix $k\in \mathbb{N}$ and $(\hat{x}, \hat{u})\in Z$. It follows from 
\eqref{equation 4.1} that
\begin{align}
\frac{x^{k}-p^{k+1}}{\tau}-L^{*}u^{k+1}-Cx^k\in & \:
A\,p^{k+1}\nonumber\\ 
\frac{u^{k}-u^{k+1}}{\gamma}+L(x^{k}+ p^{k}-x^{k-1})-D^{-1}u^{k}\in&\: 
B^{-1}u^{k+1}.
\label{equation 1.13}
\end{align}
Since $A$ and $B^{-1}$ are maximally monotone operators 
\cite[Theorem 20.25]{Livre1}, we deduce from \eqref{e:Z0} that
\begin{align}
\scal{\frac{x^{k}-p^{k+1}}{\tau}-L^{*}(u^{k+1}-\hat{u})}{p^{k+1}-\hat{x}}
-\scal{Cx^{k}-C\hat{x}}{p^{k+1}-\hat{x} } 
&\nonumber\\+\scal{\frac{u^{k}-u^{k+1}}{\gamma}+L(x^{k}+ 
p^{k}-x^{k-1}-\hat{x})}{u^{k+1}-\hat{u}}&\nonumber\\- 
\scal{D^{-1}u^{k}-D^{-1}\hat{u}}{u^{k+1}-\hat{u}} & \geq 0.
\label{equation 1.14}
\end{align}
Hence, it follows from 
\cite[Lemma~2.12(i)]{Livre1} that
\begin{align}
\frac{\| x^{k}-\hat{x}\|^2 }{\tau}+\frac{\| u^{k}-\hat{u}\|^2}{\gamma} 
\geq&\frac{\| x^{k}-p^{k+1}\|^2}{\tau}+\frac{\| 
p^{k+1}-\hat{x}\|^2}{\tau}+\frac{\| u^{k+1}-u^{k}\|^2}{\gamma} \nonumber \\ &+\frac{\| 
u^{k+1}-\hat{u}\|^2}{\gamma}+2\langle 
L(p^{k+1}-\hat{x})\mid u^{k+1}-\hat{u}\rangle\nonumber\\&-2\langle 
L(x^{k}+p^{k}-x^{k-1}-\hat{x})\mid 
u^{k+1}-\hat{u}\rangle\nonumber\\&+2\scal{D^{-1}u^{k}-D^{-1}\hat{u}}{u^{k+1}-\hat{u}}\nonumber\\&+2\scal{Cx^{k}-C\hat{x}}{p^{k+1}-\hat{x}}.
\label{equation 1.15}
\end{align}
Moreover, Cauchy-Schwartz inequality yields
\begin{align}
&\langle L(p^{k+1}-\hat{x})\mid u^{k+1}-\hat{u}\rangle-\langle 
L(x^{k}+p^{k}-x^{k-1}-\hat{x})\mid 
u^{k+1}-\hat{u}\rangle\nonumber\\
&\hspace{2cm}=\langle L(p^{k+1}-\hat{x})\mid 
u^{k+1}-\hat{u}\rangle-\langle 
L(x^{k}-\hat{x})\mid u^{k+1}-\hat{u}\rangle\nonumber\\ 
&\hspace{2.5cm}-\langle L(p^{k}-x^{k-1})\mid 
u^{k+1}-\hat{u}\rangle\nonumber\\
&\hspace{2cm}=\langle L(p^{k+1}-x^{k})\mid 
u^{k+1}-\hat{u}\rangle-\langle 
L(p^{k}-x^{k-1})\mid 
u^{k+1}-\hat{u}\rangle\nonumber\\
&\hspace{2cm}=
\langle L(p^{k+1}-x^{k})\mid u^{k+1}-\hat{u}\rangle-\langle 
L(p^{k}-x^{k-1})\mid u^{k+1}-u^{k}\rangle\nonumber\\ 
&\hspace{2.5cm}-\langle L(p^{k}-x^{k-1})\mid 
u^{k}-\hat{u}\rangle\nonumber\\ 
&\hspace{2cm}\geq
\langle L(p^{k+1}-x^{k})\mid u^{k+1}-\hat{u}\rangle-\| L\| \| 
p^{k}-x^{k-1}\|  \| u^{k+1}-u^{k}\|\nonumber\\
&\hspace{2.5cm}-\langle L(p^{k}-x^{k-1})\mid 
u^{k}-\hat{u}\rangle\\
&\hspace{2cm}\geq
\langle L(p^{k+1}-x^{k})\mid u^{k+1}-\hat{u}\rangle-
\nu\| L\|^2 \| p^{k}-x^{k-1}\|^2\nonumber\\
&\hspace{2.5cm} -\frac{1}{\nu}\| 
u^{k+1}-u^{k}\|-\langle L(p^{k}-x^{k-1})\mid 
u^{k}-\hat{u}\rangle,
\label{equation 1.16}
\end{align}
for every $\nu>0$.
In addition, the cocoercivity of $C$ and $D^{-1}$ and the inequality 
$ab \leq \mu 
a^{2}+\frac{b^{2}}{4 \mu}$ yield
\begin{align}\left\langle Cx^{k}-C\hat{x}\mid 
p^{k+1}-\hat{x}\right\rangle &=\left\langle Cx^{k}-C\hat{x}\mid 
p^{k+1}-x^{k}\right\rangle+\left\langle Cx^{k}-C\hat{x}\mid 
x^{k}-\hat{x}\right\rangle\nonumber \\ & 
\geq-\|Cx^{k}-C\hat{x}\|\|p^{k+1}-x^{k}\|+\mu\|Cx^{k}-C\hat{x}\|^{2} 
\nonumber\\ & \geq-\frac{\left\|p^{k+1}-x^{k}\right\|^{2}}{4 \mu},
\label{equation 1.17}
\end{align}
and, analogously, $\langle D^{-1}u^{k}-D^{-1}\hat{u} \mid 
u^{k+1}-\hat{u} \rangle\geq-\frac{\left\|u^{k+1}-u^{k}\right\|^{2}}{4 
\delta}$.
\\ \\ Let $(\hat{x},\hat{u})\in Z$ and let $\mathscr{X}=(\mathcal{X}_{k})_{k\in\mathbb{N}}$ be a sequence of sub-sigma-algebras  of $\mathcal{X}$ such that,  for every $ k\in\mathbb{N},\hspace{0.2cm} \mathcal{X}_{k}=\sigma(x^0,...,x^k)$.
It  follows from \eqref{equation 4.1}, the linearity of conditional 
expectation, and the mutual independence of 
$(\epsilon_k)_{k\in\mathbb{N}}$  that
\begin{align}
\label{e:aux2}
\mathbb{E}(\| x^{k+1}-\hat{x}\|^2\mid \mathcal{X}_k 
)&=\mathbb{E}\bigg(\sum_{i\in 
I_{k+1}}\mathds{1}_{\{\epsilon_{k+1}=i\}}\cdot\| 
T_{i}p^{k+1}-\hat{x}\|^2 \mid 
\mathcal{X}_k\bigg)\nonumber\\&=\pi^{0}_{k+1}\|p^{k+1}-
\hat{x}\|^2+\sum_{i\in 
I_{k+1}^{0}}\pi^{i}_{k+1}\|T_{i}p^{k+1}-\hat{x}\|^2,
\end{align}
where, for every $k\in\NN$, 
$I_{k}^{0}:=I_{k}\backslash \{0\}$.  Moreover, since, for every 
$i\in\{1,\ldots,m\}$, $T_i$ is $\alpha_i-$averaged nonexpansive, where
$\alpha_i\in\left]0,1\right[$, and $\hat{x}\in \cap_{i\in I}\Fix T_{i}$,
we have
\begin{equation}
\|T_{i}p^{k+1}-\hat{x}\|^2\le 
\|p^{k+1}-\hat{x}\|^2-\left(\dfrac{1-\alpha_i}{\alpha_i}\right)\|T_{i}p^{k+1}-
 p^{k+1}\|^2.
\end{equation}
Hence, since, for every $k\in\NN$, 
$\sum_{i\in 
I_k}\pi^{i}_{k}=1$, 
from \eqref{e:aux2} we deduce, $\mathbb{P}-$a.s.
\begin{align}
\mathbb{E}(\| x^{k+1}-\hat{x}\|^2\mid \mathcal{X}_k 
)&+\sum_{i\in 
I_{k+1}^{0}}\Frac{\pi^{i}_{k+1}\left(1-\alpha_i\right)}{\alpha_i}\| 
T_{i}p^{k+1}- p^{k+1}\|^2 \leq \| 
p^{k+1}-\hat{x}\|^2.\hspace{0.2cm}\hspace{1,5cm} 
\label{equation 4.4}
\end{align}
Replacing \eqref{equation 1.16}, \eqref{equation 1.17}, and  
\eqref{equation 4.4} in \eqref{equation 1.15} we have, 
$\mathbb{P}-$a.s.
\begin{align}
\frac{\| x^{k}-\hat{x}\|^2 }{\tau}+\frac{\| u^{k}-\hat{u}\|^2}{\gamma} 
&\geq\frac{1}{\tau}\mathbb{E}(\| 
x^{k+1}-\hat{x}\|^2\mid \mathcal{X}_k 
)+\left(\frac{1}{\tau}-\frac{1}{2\mu}\right)\| x^{k}-p^{k+1}\|^2\nonumber\\
&\hspace{0.5cm}+\frac{\| 
u^{k+1}-\hat{u}\|^2}{\gamma}
+\left(\frac{1}{\gamma}
-\frac{1}{2\delta}-\Frac{1}{\nu}\right)\|u^{k+1}-u^{k}\|^2
\nonumber\\
&\hspace{0.5cm}
+2\langle 
L(p^{k+1}-x^{k})\mid u^{k+1}-\hat{u}\rangle
\nonumber\\
&\hspace{0.5cm}-2\langle 
L(p^{k}-x^{k-1})\mid u^{k}-\hat{u}\rangle-\nu\|
L\|^2 \| p^{k}-x^{k-1}\|^2 
\nonumber\\
&  \hspace{0.5cm}
+\sum_{i\in 
I_{k+1}^{0}}\Frac{\pi^{i}_{k+1}\left(1-\alpha_i\right)}{\tau\alpha_i}\| 
T_{i}p^{k+1}- p^{k+1}\|^2.
\label{equation 4.5}
\end{align}
In particular, if we consider
$\nu:=2\left(\frac{1}{\gamma}-\frac{1}{2\delta}+\frac{2\mu\tau\| 
L\|^2}{2\mu-\tau}\right)^{-1}$ and
$\rho:=\frac{1}{2}\left(\frac{1}{\gamma}-\frac{1}{2\delta}-\frac{2\mu\tau\|
L\|^2}{2\mu-\tau}\right)>0$, we have $\nu \| 
L\|^2=\left(\frac{1}{\tau}-\frac{1}{2\mu}\right)\left(1-\nu\rho\right)$ 
and by  \eqref{equation 4.5} we obtain, $\text{$\mathbb{P}$-a.s.}$
\begin{align}
\frac{\| x^{k}-\hat{x}\|^2 }{\tau}+\frac{\| 
u^{k}-\hat{u}\|^2}{\gamma}& 
\geq\frac{1}{\tau}\mathbb{E}(\| 
x^{k+1}-\hat{x}\|^2\mid \mathcal{X}_k 
)+\left(\frac{1}{\tau}-\frac{1}{2\mu}\right)\| 
x^{k}-p^{k+1}\|^2\nonumber\\&\hspace{0.5cm}+\frac{\| 
u^{k+1}-\hat{u}\|^2}{\gamma}+2\langle L(p^{k+1}-x^{k})\mid 
u^{k+1}-\hat{u}\rangle \nonumber\\
&\hspace{0.5cm} -2\langle L(p^{k}-x^{k-1})\mid 
u^{k}-\hat{u}\rangle
\nonumber\\
&\hspace{0.5cm}-\left(\frac{1}{\tau}-\frac{1}{2\mu}\right)
\|p^{k}-x^{k-1}\|^2\nonumber\\
&\hspace{0.5cm}+\rho\|u^{k+1}-u^{k}\|^2 
+\nu\rho\left(\frac{1}{\tau}-\frac{1}{2\mu}\right)
\|p^{k}-x^{k-1}\|^2\nonumber\\
&\hspace{0.5cm}+\sum_{i\in 
I_{k+1}^{0}}\Frac{\pi^{i}_{k+1}\left(1-\alpha_i\right)}{\tau\alpha_i}\|
T_{i}p^{k+1}- p^{k+1}\|^2.
\label{equation 4.6}
\end{align}
Now, let us consider the
self-adjoint linear operator 
$\boldsymbol{V}\colon\HH\oplus\GG
\oplus\HH\to\HH\oplus\GG\oplus\HH$ defined by
\begin{equation}
\boldsymbol{V}\colon (x,u,p)\mapsto
\left(\frac{x}{\tau}, \frac{u}{\gamma} 
+Lp,\left(\frac{1}{\tau}-\frac{1}{2\mu}\right)p+L^*u\right).
\end{equation}
Note that, for every $\boldsymbol{x}:=(x,u,p)\in 
\HH\oplus\GG\oplus\HH$, from 
\eqref{equation 4.1.0} we 
deduce
\begin{align}
\pscal{\boldsymbol{x}}{\boldsymbol{V}\boldsymbol{x}}&=\frac{\|x\|^2}{\tau}
+\frac{\|u\|^2}{\gamma}+2\scal{Lp}{u}+
\left(\frac{1}{\tau}-\frac{1}{2\mu}\right)\|p\|^2\label{e:aux}\\
&\ge \frac{\|x\|^2}{\tau}
+\frac{\|u\|^2}{\gamma}-2\|L\|\|p\|\|u\|+
\left(\frac{1}{\tau}-\frac{1}{2\mu}\right)\|p\|^2\nonumber\\
&\ge \frac{\|x\|^2}{\tau}
+\frac{\|u\|^2}{\gamma}-2\|L\|\|p\|\|u\|+
(\gamma\|L\|^2+\varepsilon)\|p\|^2\nonumber\\
&\ge \frac{\|x\|^2}{\tau}+
\|u\|^2\left(\frac{1}{\gamma}-\frac{1}{e}\right)+
((\gamma-e)\|L\|^2+\varepsilon)\|p\|^2,\nonumber
\end{align}
where 
$\varepsilon=(\frac{1}{\tau}-\frac{1}{2\mu})-\gamma\|L\|^2>0$
and $e>0$ is arbitrary.
Hence, by taking $\gamma <e< \gamma+\varepsilon/\|L\|^2$ we 
deduce that  $\boldsymbol{V}$ is strongly monotone
in $\HH\oplus\GG\oplus\HH$. Define the scalar product 
$\pscal{\cdot}{\cdot}_{\boldsymbol{V}}
=\pscal{\cdot}{\boldsymbol{V}\cdot}$
and set $\HHH$ be the real Hilbert space
$\HH\times\GG\times\HH$ endowed with this scalar product.
We denote by 
$||\cdot||_{\boldsymbol{V}}
=\sqrt{\pscal{\cdot}{\cdot}_{\boldsymbol{V}}}$ the associated 
norm.
Note that, for every $k\in\NN$,
\begin{equation}
\left(\frac{1}{\tau}-\frac{1}{2\mu}\right)\| 
p^{k+1}-x^{k}\|^2+2\langle L(p^{k+1}-x^{k})\mid 
u^{k+1}-\hat{u}\rangle+\frac{\| u^{k+1}-\hat{u}\|^2}{\gamma}
\end{equation}
is $\mathcal{X}_k-$measurable.
 Therefore, by defining, for every $k\in\NN$,
 $\boldsymbol{x}^k=(x^k,u^k,p^k-x^{k-1})\in\HHH$
 and set $\hat{\boldsymbol{x}}=(\hat{x},\hat{u},0)\in Z\times\{0\}$, 
 we deduce from 
\eqref{equation 4.6} and \eqref{e:aux} that
\begin{equation}
\mathbb{E}\left(\|\boldsymbol{x}^{k+1}-
\hat{\boldsymbol{x}}\|^2_{\boldsymbol{V}} \,| \,
\mathcal{X}_k 
\right)+b_k\le \|\boldsymbol{x}^k-
\hat{\boldsymbol{x}}\|^2_{\boldsymbol{V}},
\label{equation 4.9x}
\end{equation}
where 
\begin{multline}
\label{equation 4.8}
b_{k}:=\rho\|u^{k+1}-u^{k}\|^2+\nu\rho\left(\frac{1}{\tau}-\frac{1}{2\mu}\right)\|
p^{k}-x^{k-1}\|^2\\
\hspace{1cm}+\sum_{i\in 
I_{k+1}^{0}}\frac{\pi^{i}_{k+1}\left(1-\alpha_i\right)}{\tau\alpha_i}\| 
T_{i}p^{k+1}- p^{k+1}\|^2
\end{multline}
defines a sequence in $\ell_+(\mathscr{X})$.
By denoting $\boldsymbol{Z}=Z\times\{0\}$, we deduce from 
Lemma~\ref{l: lemma 1.4} that, 
$\mathbb{P}-$a.s.,
\begin{align}
\sum_{k\in\NN} \|u^{k+1}-u^{k}\|^2<+\infty, \hspace{0.2cm}  
\sum_{k\in\NN}\|p^{k}-x^{k-1}\|^2<+\infty,\nonumber\\
\text{and}\quad \sum_{k\in\NN}\sum_{i\in 
	I^{0}_{k+1}}\frac{\pi^{i}_{k+1}\left(1-\alpha_i\right)}{\tau\alpha_i}
\|T_{i}p^{k+1}-	 p^{k+1}\|^2<+\infty.
\label{equation 4.12}
\end{align}
Let $\widetilde{\Omega}$ be the set such that 
$\mathbb{P}(\widetilde{\Omega})=1$ and \eqref{equation 4.12}
holds, and fix $w\in\widetilde{\Omega}$.
Let $\boldsymbol{x}(w):=(x(w),u(w),z(w))
\in\mathfrak{W}(\boldsymbol{x}^{k}(w))_{k\in 
\mathbb{N}}$, say 
\begin{equation}
\boldsymbol{x}^{k_n}(w)
=(x^{k_{n}}(w),u^{k_{n}}(w),p^{k_n}(w)-x^{k_n-1}(w))\weakly 
(x(w),u(w),z(w)).
\end{equation}
Note that, \eqref{equation 4.12} yields
$z(w)=0$ and, therefore, in order to prove that
$\boldsymbol{x}(w)\in\boldsymbol{Z}=Z\times\{0\}$ it is enough to 
prove that $(x(w),u(w))\in Z$.
By defining 
\begin{align}
\label{equation 2.17}
M:(x, u)&\mapsto (Ax + L^{*}
u) \times (B^{-1}u -Lx)\nonumber\\
Q :(x, u) &\mapsto\left(Cx, D^{-1}u\right),
\end{align}   and 
\begin{align}
\label{equation 2.18}
y^k :=&\frac{x^k-p^{k+1}}{\tau}+Cp^{k+1}-Cx^{k}\nonumber\\
v^k 
:=&\frac{u^k-u^{k+1}}{\gamma}+L(x^{k}-p^{k+1}+p^k-x^{k-1})+ 
D^{-1}u^{k+1}-D^{-1}u^{k},
\end{align}
we obtain from \eqref{equation 1.13} that
\begin{align}
(y^k, v^k)\in (M+Q)(p^{k+1},u^{k+1}).
\label{equation 4.16}
\end{align}

Moreover, from  \cite[Proposition 
2.7(iii)]{Siopt1} we have that $M$ is maximally monotone. Since $D$ 
is $\delta-$strongly monotone, 
\cite[Proposition~22.11(ii)]{Livre1} yields its surjectivity. 
Hence,  since $C$ is $\mu-$cocoercive, we deduce that $Q$ is 
$\min\{\mu,\delta\}-$cocoercive and, from \cite[Theorem 
21.1]{Livre1}, that  $\dom 
Q=\HH\times\GG$. Therefore, by \cite[Corollary 25.5(i)]{  Livre1} we 
conclude that $M+Q$ is maximally monotone.
Moreover, it follows from \eqref{equation 4.12}, 
$(x^{k_{n}}(w),u^{k_{n}}(w))\rightharpoonup (x(w),u(w))$,
and from the uniform continuity of $L$, $C$, and $D^{-1}$ that 
$(y^{k_n}(w),v^{k_n}(w))\rightarrow0$ and 
$(p^{k_{n}+1}(w),u^{k_{n}+1}(w))\rightharpoonup(x(w),u(w))$. 
Therefore, we 
deduce from \eqref{equation 4.16} and 
\cite[Proposition~20.37]{Livre1} that 
$(x(w),u(w))\in~Z_0$ and, in the case \ref{A:algorithm 4.10}, 
$(x(w),u(w))\in~Z$ and the result 
follows from Lemma~\ref{l: lemma 1.4}.

In order to prove the convergence under the assumption in 
\ref{A:algorithm 4.1i}, note that
it follows from \eqref{equation 4.1}, \eqref{A:algorithm 4.1ii}, and 
\eqref{equation 4.12} that, 
for every $k\in \NN$,
\begin{align}
\|x^{k}(w)-x^{k-1}(w)\|&=
\|T_{\epsilon_{k}(w)}p^{k}(w)-x^{k-1}(w)\|\nonumber\\
&\le\|T_{\epsilon_{k}(w)}p^{k}(w)-p^{k}(w)\|
+\|p^{k}(w)-x^{k-1}(w)\|\nonumber\\
&\le\sum_{i\in 
	I^0_{k}}
\|T_{i}p^{k}(w)-p^{k}(w)\|+\|p^{k}(w)-x^{k-1}(w)\|\label{e:interm}\\
&\le\frac{\overline{\alpha}}{\zeta (1-\overline{\alpha})}\sum_{i\in 
I^0_{k}}\frac{\pi^{i}_{k+1}\left(1-\alpha_i\right)}{\tau\alpha_i}
\|T_{i}p^{k}(w)-p^{k}(w)\|\nonumber\\
&\hspace{5cm}+\|p^{k}(w)-x^{k-1}(w)\|\nonumber\\
&\to 0,\label{e:auxs}
\end{align}
where $\overline{\alpha}=\max_{i=1,\ldots m}\alpha_i$.
Now, let $N\in\NN\smallsetminus\{0\}$ 
be the integer provided by assumption \ref{A:algorithm 4.1i}, fix  $i\in 
I$, and fix $n\in\NN$. Therefore, \ref{A:algorithm 4.1i} ensures the 
existence of $j_n$ such that $i\in I_{j_n}$ and
$k_{n}+1\leq j_n\leq 
k_{n}+N$ and, from \eqref{e:auxs}
we deduce
\begin{align}
\|x^{j_n}(w)-x^{k_n}(w)\|&\leq 
\sum_{l=k_n}^{j_n-1}\|x^{l+1}(w)-x^{l}(w)\|\nonumber\\
&\leq 
\sum_{l=k_n}^{k_n+N-1}\|x^{l+1}(w)-x^{l}(w))\|\nonumber\\
&\to 0.
\label{equation  4.14}
\end{align}
Hence, from 
$x^{k_n}(w)\weakly x(w)$,
\eqref{equation  4.14}, \eqref{e:interm}, \eqref{e:auxs}, and 
\eqref{equation 4.12}
 we deduce that 
\begin{equation}
p^{j_n}(w)\weakly x(w)\:\: \text{ and  } 
(\Id-T_{i})p^{j_{n}}(w)\rightarrow 0.
\end{equation}
Therefore, since $\Id-T_{i}$ is maximally 
monotone \cite[Example 20.29]{Livre1} and its graph is
 closed in the weak-strong topology \cite[Proposition 20.38]{Livre1}, 
 we deduce $x(w)\in \Fix(T_{i})$.
Since $i\in I$ is arbitrary, we conclude $x(w)\in \bigcap_{i=0}^m 
\Fix\left(T_i\right)=S$
and the result follows from Lemma~\ref{l: lemma 1.4}.
\end{proof}

\begin{remark}
\label{rem:1}
In this remark, we explore the flexibility of our formulation.
\begin{enumerate}	
	\item 
\label{rem:1opti}
\textit{Random projections in convex optimization:} 
Consider the context of primal-dual convex optimization problems
\eqref{equation 1.1}-\eqref{equation 1.2}. For every 
$i\in\{1,\ldots,m\}$, suppose that $T_i=P_{S_i}$, where  $S_i$ is a 
nonempty 
closed 
convex subset of 
$\HH$. Then, $S:=\bigcap_{i=1}^m 
S_i\neq\varnothing$
 and \eqref{equation 4.1} reduces to 
\begin{align}
(\forall k\in\NN)\hspace{0.2cm}
&\left\lfloor 
\begin{array}{ll}
u^{k+1}=\prox_{\gamma g^{*}}(u^k+\gamma 
(L\bar{x}^k-\nabla\ell^{*}(u^{k})))\\
p^{k+1}=\prox_{\tau f}(x^k-\tau( L^*u^{k+1}+\nabla H(x^k)))\\
x^{k+1}=P_{S_{\epsilon_{k+1}}}\,p^{k+1}\\
\bar{x}^{k+1}=
x^{k+1}+ p^{k+1}-x^{k},
\label{equation 4.1.1.1}
\end{array}
\right.
\end{align}
which solves \eqref{equation 1.1}-\eqref{equation 1.2} if 
$(\epsilon_k)_{k\ge 1}$ satisfies \ref{A:algorithm 4.1i}.
This algorithm will be revisited in the application studied in 
Section~\ref{sec:traffic}.

	\item 
\label{rem:1Bernoulli} \textit{Primal-dual with cyclic 
	Bernoulli random activations:}
Suppose that, for every $k\in\NN$,
$I_{k}=\{0,i(k)\}$, where $i\colon \NN\to I$ is a function such that,
for every $n\in\NN$, $i(n,\ldots,n+N-1)=\{1,\ldots,m\}$ and $N$ is 
defined in \ref{A:algorithm 4.1i}.
Moreover, define $(e_k)_{ k \in \mathbb{N}}$ as the sequence of 
independent 
$\{0,1\}-$valued random variables such that, for every $k\in\NN$, 
$e_{k}^{-1}(\{0\})=\epsilon_{k}^{-1}(\{0\})$. Hence, we have
that, for every $n\in\NN$,
\begin{equation}
\label{e:hipot}
\cup_{k=n}^{n+N-1}I_k=\{0\}\cup i(n,\ldots,n+N-1)=I,
\end{equation}
\ref{A:algorithm 4.1i} holds, and the activation step in	
\eqref{equation 4.1} is equivalent to
\begin{align}
\label{e:reduce}
(\forall k\in\NN)\quad 
x^{k+1}=p^{k+1}+e_{k+1}(T_{i(k+1)}\,p^{k+1}-p^{k+1})
=T_{\epsilon_{k+1}}\,p^{k+1},
\end{align}
where $T_{0}=\Id$. Therefore, we deduce that 
a particular instance of \eqref{equation 4.1}is the primal-dual method 
with random Bernoulli
activations as it is used, e.g., in \cite{Siopt6}. In 
particular, if 
$i\colon 
k\mapsto (k\!\!\mod m)+1$, the random Bernoulli activation is 
applied 
over a cyclic order of the operators $(T_i)_{1\leq i\leq m}$.
This is a generalization of the primal-dual method with a 
priori 
information developed in \cite{BA-Lopez}. In addition, if for every $ k 
\in \NN$, $\epsilon_{k}^{-1}\left(\{0\}\right)=\emp$, the 
activations become deterministic and \eqref{A:algorithm 4.1ii} holds. 
In this context, we recover 
\cite{BA-Lopez} by setting $m=1$ and $T_1=T$.

\item \textit{Projections onto convex sets}:
\label{rem:1POCS}
In the context of Remark~\ref{rem:1}\ref{rem:1opti}, suppose that 
$f= 
g=h=L=0$, $\ell=\iota_{\{0\}}$,  
for every $k\in\NN$,
$I_k=\{0,
(k\hspace{0.1cm}\mathrm{mod}\hspace{0.1cm}m)+1\}$, and 
that 
$\epsilon_{k}^{-1}\left(\{0\}\right)=\emp$. Then  
$\prox_{f}(x)=x$, 
$\prox_{ g^{*}}(x)=x-\prox_{g}(x)=0$ 
\cite[Proposition 24.8(ix)]{Livre1},
\ref{A:algorithm 4.1i}
holds,  the activation scheme in \eqref{equation 4.1.1.1} 
reduces to
\begin{align}
(\forall k\in\NN)\hspace{0.2cm}
x^{k+1}=P_{S_{k\,\mathrm{ mod }\,m+1}}x^{k},
\label{equation 2.19}
\end{align}
and we recover the convergence of the method in
\cite[Corollary 5.26]{Livre1}. We can also recover more general 
projection schemes as, for instance, cyclic activations \cite{plc97},
which inspire condition \ref{A:algorithm 4.1i}.

		\item 
\label{rem:1kacz}
\textit{Kaczmarz algorithm:}
In the context of Remark~\ref{rem:1}\ref{rem:1POCS},
assume that $\mathcal{H}=\mathbb{R}^n$,
let $R$ be a full rank $m\times n$ matrix such that $m \leq n$, 
denote the rows of $R$ by $r_1,...,r_m\in\RR^n$,
and set  
$b=(b_1, . . . , b_m)^{\top}\in\mathbb{R}^m$. Moreover, suppose 
that 
$S=\menge{x\in\mathcal{H}}{Rx=b}=\bigcap_{i=1}^{m} 
S_i\not=\emp$, where, for every $i\in\{1,\ldots,m\}$, $S_i=\{x\in 
\RR^n\colon r_i^{\top}x=b_i\}$ is 
nonempty, closed, and convex. Then, \eqref{equation 2.19} 
reduces to
\begin{align}
(\forall k\in\NN)\hspace{0.2cm}
x^{k+1}=x^{k}+\frac{b_{i(k+1)}-
	r_{i(k+1)}^{\top} x^{k}}{\|r_{i(k+1)}\|^2}r_{i(k+1)},
\label{equation 2.20}
\end{align}
which is the Kaczmarz method proposed in  \cite{Kaczmarz} and 
its convergence follows from Theorem~\ref{A:algorithm 4.1}.

\item \textit{Randomized Kaczmarz:}
Consider the setting described in 
Remark~\ref{rem:1}\ref{rem:1kacz}
and let $(\epsilon_k)_{ k 
	\in \mathbb{N}}$ be a sequence of independent $I$-valued 
random 
variables, with $\epsilon_k^{-1}(\{0\})\equiv\emp$, and for every $i\in 
I\smallsetminus\{0\}$ 
and $k\in 
\NN$, 
$\pi_{k}^i$ is 
proportional
to $\|r_i\|^2$. Therefore, as in Remark~\ref{rem:1}\ref{rem:1kacz}, 
\eqref{equation 4.1} reduces to
\begin{align}
(\forall k\in\NN)\hspace{0.2cm}
x^{k+1}=P_{S_{\epsilon_{k+1}}}x^{k}=x^{k}+\frac{b_{\epsilon_{k+1}}- 
	r_{\epsilon_{k+1}}^{\top}x^{k}}{\|r_{\epsilon_{k+1}}\|^2}
r_{\epsilon_{k+1}},
\label{equation 3.16}
\end{align}
which is the method proposed in  \cite{Strohmer-Vershynin} and
its convergence is deduced from Theorem~\ref{A:algorithm 4.1}. 
Note that 
in \cite{Strohmer-Vershynin} the authors obtain exponential 
convergence rate in expectation, while our method has
$\mathbb{P}-$a.s. convergence.

\end{enumerate}

\end{remark}

\section{Application to the arc capacity expansion problem of a 
directed graph}
\label{sec:traffic}

In this section we aim at solving the traffic assignment problem with 
arc-capacity expansion on a network 
with minimal cost under uncertainty. Let  ${\cal{A}}$ be the set of 
arcs and let ${\cal{O}}$ and ${\cal{D}}$ be the sets of origin and 
destination nodes of the network, respectively. The set of 
routes from 
$o\in {\cal{O}}$ to $d\in {\cal{D}}$ is denoted by $R_{od}$ and
$R=\cup_{(o,d)\in {\cal{O}}\times{\cal{D}}} 
R_{od}$ is the set of all routes. The arc-route 
incidence matrix $N\in 
{\mathbb{R}}^{|{\cal{A}}|\times |R|}$ is defined by 
$N_{a,r}=1$, if arc $a$ belongs to the 
route $r$, and $N_{a,r}=0$, otherwise.  

The uncertainty is modeled by a finite set $\Xi$ of
possible scenarios. For every scenario $\xi\in\Xi$, $p_{\xi}\in 
[0,1]$ is its probability of occurrence, 
$h_{od,\xi}\in {\mathbb{R}_+}$ is 
the forecasted demand from $o\in {\cal{O}}$ to $d\in 
{\cal{D}}$, 
$c_{a,\xi}\in {\mathbb{R}}_{+}$ 
is the corresponding capacity of the arc $a\in {\cal{A}}$,
$t_{a,\xi}\colon 
{\mathbb{R}}_{+}\rightarrow {\mathbb{R}}_{+}$ is an 
increasing and $\beta_{a,\xi}-$Lipschitz continuous travel time 
function on arc $a\in {\cal{A}}$, for some $\beta_{a,\xi}>0$, 
and the variable $f_{r,\xi}\in 
{\mathbb{R}_+}$ 
stands for the flow in route $r\in R$. 

In the problem of this section, we consider the expansion of 
flow capacity at each arc in order to improve the efficiency of the 
network operation. We model this decision 
making process in a two-stage stochastic problem. The first stage 
reflects the 
investment in capacity and the second corresponds to the 
operation of the network in an uncertain environment. 

In order to solve this problem, we take a non-anticipativity approach 
\cite{Xiaojun Chen}, 
letting our first stage decision variable depend on the scenario
and	imposing a non-anticipativity constraint. We denote by 
$x_{a,\xi}\in\RR_+$ the variable of capacity expansion on arc 
$a\in\mathcal{A}$ in scenario $\xi\in\Xi$ and the non-anticipativity 
condition is defined by the constraint 
\begin{equation*}
W=\big\{x\in\RR^{|\mathcal{A}|\times|\Xi|}\,:\,(\forall (\xi,\xi')\in 
\Xi^{2})\quad x_{\xi}=x_{\xi'}\big\},
\end{equation*}
where $x_{\xi}\in\RR^{|\mathcal{A}|}$ is
the vector of capacity expansion for scenario $\xi\in\Xi$ and we denote
$f_{\xi}\in\RR^{|R|}$ analogously.
We restrict the capacity expansion variables by imposing, for every 
$a\in \mathcal{A} $ and $\xi\in\Xi$,
$x_{\xi}\in  M:=\bigtimes_{a\in\mathcal{A}}[0,M_a]\subset 
\RR^{|\mathcal{A}|}$, where $M_a>0$ represents the upper 
bound of 
capacity expansion, for every $a\in \mathcal{A}$. Additionally,
we model the investment cost of expansion via a quadratic function
defined by a symmetric positive definite matrix
$Q\in \RR^{|{\cal{A}}|\times |{\cal{A}}|}$.
\begin{problem}
\label{prob_transporte}
The problem is to
\begin{equation*}
\min_{(x,f)\in  (W\cap M^{|\Xi|})\times 
{\mathbb{R}}_{+}^{|R||\Xi|}}  \displaystyle\sum_{\xi\in \Xi} 
p_{\xi}\left[\displaystyle\sum_{a\in {\cal{A}}} 
\displaystyle\int_{0}^{(Nf_{\xi})_a} 
t_{a,\xi}(z)dz+\frac{1}{2}x_{\xi}^{\top}Qx_{\xi}\right]
\end{equation*}
\begin{align}
\text{s.t.}				\hspace{1cm} (\forall \xi\in \Xi)(\forall a\in 
{\cal{A}})\hspace{1cm}
(Nf_{\xi})_a 
-x_{a,\xi}&\leq c_{a,\xi},
\label{rest2_prob_tr}\\
(\forall \xi\in \Xi)(\forall (o,d)\in {\cal{O}}\times{\cal{D}})\hspace{1cm}
\displaystyle\sum_{r\in R_{od}} f_{r,\xi} &=h_{od,\xi},
\label{rest1_prob_tr}
\end{align}
under the assumption that the solution set $Z_1$ 
is nonempty.\end{problem}The first term of the objective function in 
Problem~\ref{prob_transporte} 
represents the expected operational cost of the network. 
The optimality conditions of the optimization problem with this 
objective cost related to the
pure traffic assignment problem, defines a Wardrop equilibrium  
\cite{Beckmann}. The second term in the objective function is the 
expansion investment cost. Constraints in \eqref{rest2_prob_tr} 
represent that, for every arc $a\in\mathcal{A}$, the flow cannot 
exceed the expanded 
capacity $c_{a,\xi}+x_{a,\xi}$ at each scenario $\xi\in\Xi$,
while \eqref{rest1_prob_tr} are the demand constraints.

\subsection{Formulation and Algorithms}

Note that Problem~\ref{prob_transporte} can be equivalently 
written as
\begin{align*}
\label{prob_esto_aplic}\tag{$\mathcal{P}$}
\text{ find  }(x,f)\in S\cap \Bigg(\underset{(x,f)\in\HH}{\argmin\,}
F(x,f)+G(L(x,f))+H(x,f)\Bigg),
\end{align*}
where 
\begin{equation}
\label{e:defsform1}
\begin{cases}
\HH:=\RR^{|\mathcal{A}||\Xi|}\times\RR^{|R||\Xi|}\\[2mm]
(\forall \xi\in\Xi)\quad V^{+}_\xi:=\Big\{f\in {\mathbb{R}}_{+}^{|R|}\,:\, 
(\forall (o,d)\in {\cal{O}}\times{\cal{D}})\,\, \sum_{r\in 
R_{od}} 
f_{r} =h_{od,\xi}\Big\}\\[3mm]
\Lambda:=(M^{|\Xi|}\cap W)\times 
\Big(\bigtimes_{\xi\in\Xi} V_{\xi}^{+}\Big)\\[3mm]
F:=\iota_{\Lambda}\\[2mm]
(\forall \xi\in\Xi)\quad H_{\xi}:=\left\{(x,u)\in 
{\mathbb{R}}^{|{\cal{A}}|}\times{\mathbb{R}}^{|{\cal{A}}|}\,:\, 
(\forall 
a\in {\cal{A}})\,\, u_{a} -x_{a}\leq c_{a,\xi}\right\}\\[3mm]
G:=\iota_{\bigtimes_{\xi \in \Xi} H_{\xi}}\\[2mm]
L\colon (x,f)\mapsto (x_{\xi},Nf_{\xi})_{\xi\in \Xi}\\[2mm]
S:=\cap_{i=1}^{m}S_{i}\supset \dom\, (G\circ L)\\[2mm]
H\colon (x,f) \mapsto \displaystyle\sum_{\xi\in \Xi} 
p_{\xi}\left[\sum_{a\in {\cal{A}}} 
\displaystyle\int_{0}^{(Nf_{\xi})_a} 
t_{a,\xi}(z)dz+\frac{1}{2}x_{\xi}^{\top}Qx_{\xi}\right],
\end{cases}
\end{equation}
and $(S_i)_{1\le i\le m}$ are nonempty closed convex sets
such that $S\ne\emp$. Note that $F$ and $G$ are lower 
semicontinuous 
convex proper 
functions, and $L$ is linear and bounded with $\|L\|\leq 
\max\left\{1,\|N\|\right\}$. Moreover, note that, since 
$(t_{a,\xi})_{a\in\mathcal{A},\xi\in\Xi}$ are increasing, $N$ is linear, 
and $Q$ is definite positive, $H$ is a separable convex function. In 
addition, by defining $$\psi_{\xi}:\RR^{|R|}\mapsto\RR^{|R|}:f\mapsto 
N^{\top}(t_{a,\xi}(N_{a}f))_{a\in\mathcal{A}},$$ simple computations 
yield
\begin{equation*}
\nabla 
H\colon (x,f)\mapsto \left(\left(p_{\xi}Qx_{\xi}\right)_{\xi\in 
\Xi},\left(\psi_{\xi}(f)\right)_{\xi\in\Xi}\right),
\end{equation*}
which  is Lipschitz continuous with constant
\begin{equation*}
\mu^{-1}=\max_{\xi \in \Xi} \left( p_{\xi}
\max\Big\{\|Q\|,\|N\|^{2}\max_{a\in {\cal{A}}} 
\beta_{a,\xi}\Big\}\right).
\end{equation*}
Altogether, \eqref{prob_esto_aplic} is a particular case of 
\eqref{equation 1.1}. Assume that the following
Slater condition 
\begin{multline}
\exi(\hat{x},\hat{f})\in\Lambda
\quad \text{such that}\\
\left(\forall \xi\in\Xi\right)
\left(\forall a\in 
\mathcal{A}\right)\quad\sum_{r\in 
R}N_{a,r}\hat{f}_{r,\xi}-\hat{x}_{a,\xi}< 
c_{a,\xi}
\end{multline}
holds.
Then by \cite[Proposition 27.21]{Livre1} the qualification condition 
\eqref{equation 1.10} is satisfied and, therefore, 
\eqref{prob_esto_aplic} is a 
particular case of Problem~\ref{problema 1}.

Observe that the a priori 
information $S$ is redundant with the objective function, because 
$S=\cap_{i=1}^{m}S_{i}\supset \dom\,(G\circ L)$. 
We will show in Section~\ref{sec:numexp} that this redundant 
formulation has important numerical advantages.
In what follows, we exploit the splittable structure of $S$ and provide
an application of the primal-dual splitting with random projection 
detailed in \eqref{equation 4.1.1.1} in order to solve 
Problem~\ref{prob_transporte}. The next result is a consequence
of Theorem~\ref{A:algorithm 4.1}\ref{A:algorithm 4.10} applied to the 
context of 
\eqref{prob_esto_aplic} as in Remark~\ref{rem:1}\ref{rem:1opti}.
We denote by 
$P_\xi:
\RR^{|\mathcal{A}||\Xi|}\times\RR^{|R||\Xi|}
\mapsto\RR^{|\mathcal{A}|}\times\RR^{|R|}\colon \left(x,f\right)\mapsto 
\left(x_{\xi},f_{\xi}\right)$
the orthogonal projection onto the scenario $\xi\in\Xi$.

\begin{corollary}
\label{C:Corrollary 5.2}
 \label{t:Algoritmo 4}
 Consider the setting of Problem \eqref{prob_esto_aplic}. Let $\gamma 
 >0$ and let $\tau\in\left ]0,2\mu\right[$ be such that
 \begin{equation}
 	\max\{1, \|N\|^2\}< 
 	\frac{1}{\gamma}\left(\frac{1}{\tau}-\frac{1}{2\mu}\right).
 \end{equation}
Let $(u^0,v^0)\in\GG$, let 
 $(x^0,f^0),\hspace{0.1cm}(\overline{x}^0,\overline{f}^0)\in 
 \mathcal{H}^2$ be such that 
 $(x^0,f^0)=(\overline{x}^0,\overline{f}^0)$ and, 
 set $I=\{0,1,...,m\}$. 
 Let $(\epsilon_k)_{ k 
 	\in \mathbb{N}}$ be a sequence of independent random variables 
 	such  that,  for every 
 $k\in\NN$, $\epsilon_k(\Omega)= I_{k}\subset I$ and 
  consider the following routine\\
  
 \scalebox{0.94}{\parbox{\linewidth}{\begin{align}
(\forall k\in\NN)\hspace{0.2cm}
&\left\lfloor 
\begin{array}{ll}\text{For every }\xi\text{ in }\Xi\\ \left\lfloor 
\begin{array}{l}
\left(\widetilde{u}_{\xi}^{k+1},\widetilde{v}_{\xi}^{k+1}\right)
=\left(u^{k}_{\xi}+\gamma\bar{x}_{\xi}^k,\hspace{0.1cm}v_{\xi}^{k}
+\gamma N\bar{f}^k_{\xi}\right)\\ 
\left(u_{\xi}^{k+1},v_{\xi}^{k+1}\right)
=\left(\widetilde{u}^{k+1}_{\xi},\hspace{0.1cm}
\widetilde{v}^{k+1}_{\xi}\right)-\gamma 
P_{H_{\xi}}\left(\gamma^{-1}\left(\widetilde{u}^{k+1}_{\xi},
\widetilde{v}^{k+1}_{\xi}\right)\right)\\(\widetilde{p}_{\xi}^{k+1},
\widetilde{g}_{\xi}^{k+1})=\left(x^k_{\xi}-\tau 
(u^{k+1}_{\xi}+p_{\xi}Qx^k_{\xi}),\hspace{0.1cm}f^k_{\xi}-\tau( 
N^{\top}v^{k+1}_{\xi}+ p_{\xi}\psi_{\xi}(f^k_{\xi}))\right)\\ 
(p_{\xi}^{k+1},g_{\xi}^{k+1})=\left(P_{\xi}\left(P_{M^{|\Xi|}\cap 
W}\left(\widetilde{p}^{k+1}\right)\right),\hspace{0.1cm}
P_{V^{+}_{\xi}}\left(\widetilde{g_{\xi}}^{k+1}\right)\right)   
\end{array}\right.\\
(x^{k+1},f^{k+1})=P_{S_{\epsilon_{k+1}}}\left(p^{k+1}, 
g^{k+1}\right)\\
(\bar{x}^{k+1},\bar{f}^{k+1})=
\left(x^{k+1},f^{k+1}\right)+ \left(p^{k+1},g^{k+1}\right)-\left(x^{k},f^{k}\right),
\end{array}
\right.
\label{routine 4}
\end{align}}}\\

where $S_0=\HH$.
Then $((x^k,f^k))_{k\in\mathbb{N}}$ converges $\mathbb{P}-$a.s. to a 
$Z_1-$valued random variable.
 \end{corollary}

\begin{remark}
\label{R: remark6}
\leavevmode	In order to implement the algorithm in \eqref{routine 4} 
we 
need to compute the following projections:
\begin{enumerate}[ label={\arabic*.} ]
\item\label{R: remark6.1} It follows from 
\cite[Proposition~24.11]{Livre1} that
\begin{align}
(\forall \xi\in\Xi)\quad 
P_{H_{\xi}}\colon (x,u)\mapsto \left(P_{H_{a,\xi}}(x_{a,\xi}, 
u_{a,\xi})\right)_{a\in\mathcal{A}}
\end{align}
and, by applying \cite[Example 29.20]{Livre1}, we obtain, for every 
$a\in\mathcal{A}$ and $\xi\in\Xi$,
\begin{align}
P_{H_{a,\xi}}\colon (\eta,\nu)\mapsto\left\{
\begin{array}{ll}
\left(\frac{\eta+\nu-c_{a,\xi}}{2}, \frac{\eta+\nu+c_{a,\xi}}{2}\right),
& \quad 
\text{if 
}\eta-\nu+c_{a,\xi}<0;\\
\hspace{0.1cm}\left(\eta,\nu\right),  & \quad  \text{if 
}\eta-\nu+c_{a,\xi}\geq 0.
\end{array} 
\right.
\end{align}
\item\label{R: remark6.2} We deduce from 
\cite[Theorem~3.16]{Livre1} that 
\begin{align}
P_{M^{|\Xi|}\cap 
W}:\RR^{|\mathcal{A}||\Xi|}\mapsto\RR^{|\mathcal{A}||\Xi|}:x
\mapsto\big(\operatorname{mid}\left(0,\bar{x},M_a\right)
\big)_{\substack{a\in\mathcal{A}\\\xi\in\Xi}},
\end{align} where 
$\bar{x}=\frac{1}{|\Xi|}\sum_{\xi\in\Xi}x_{\xi}$ and 
$\operatorname{mid}\colon (a,b,c)\mapsto P_{[a,c]}$ is the 
median function.
\item \label{R: remark6.3}Note that  \begin{align}(\forall 
\xi\in\Xi)\hspace{0.4cm}
V^{+}_{\xi}= 
{\displaystyle\varprod_{(o,d)\in\mathcal{O}\times\mathcal{D}}} 
\underbrace{\Bigg\{ f_{od}\in\RR^{|R_{o,d}|}_{+}\biggm | \sum_{r\in 
R_{o,d}}f_{od,r}=h_{od,\xi} \Bigg\}}_{V^{+}_{od,\xi}}
\end{align}
and it follows from \cite[Proposition 24.11]{Livre1} that 
\begin{equation}
P_{V_{\xi}^{+}}:\RR^{|R|}\mapsto\RR^{|R|}:f\mapsto\Big(P_{V 
^{+}_{od,\xi}}f_{od}\Big)_{(o,d)\in\mathcal{O}\times\mathcal{D}}.\end{equation}
In order to compute the projection onto 
$V^{+}_{od,\xi}$, 
we use the algorithm 
proposed in \cite{Cominetti}, which follows from quasi-Newton method 
onto the dual of the knapsack problem.
\end{enumerate}
\end{remark}
\subsection{Selection criteria for random projections}
\label{sec:select}
Now we establish several criteria for selecting half-spaces
from the constraints as a priori information sets $(S_i)_{1\le i\le m}$.
The resulting choice of $(S_i)_{1\le i\le m}$ lead to different families of 
algorithms. In Section~\ref{sec:numexp}, we compare their numerical 
performance with respect to the base
scheme, in which no projection is applied. 

Note that the feasible set defined by constraints \eqref{rest2_prob_tr} 
in Problem~\ref{prob_transporte} can be written as
\begin{align}
(\forall 
\xi\in\Xi)\hspace{0.3cm}\bigcap_{a\in\mathcal{A}}
P_{\xi}\left(\Sigma_{a,\xi}\right),\end{align}
where 
\begin{align}
\label{ecuacion 38}  
(\forall \xi\in\Xi)(\forall 
a\in\mathcal{A})\hspace{0.3cm}\Sigma_{a,\xi}:=
\left\{(x,f)\in\RR^{|\mathcal{A}||\Xi|}\times\RR^{|R||\Xi|}\colon
N_{a}f_{\xi}-x_{a,\xi}\leq c_{a,\xi}\right\}.
\end{align} 
The a priori information sets $(S_i)_{1\le i\le m}$ are chosen as 
particular intersections of  $(\Sigma_{a,\xi})_{(a,\xi)\in 
\mathcal{A}\times\Xi}$ for guaranteeing feasibility of primal iterates.
In order to obtain simple projection implementations,
note that 
\eqref{ecuacion 38} and \cite[Proposition~29.23]{Livre1} 
yields
\begin{multline}
\label{e:ortog}
(\forall (a,a')\in\mathcal{A}^2)(\forall (\xi,\xi')\in\Xi^2)\\
 \xi\neq\xi'\quad\Rightarrow\quad 
P_{\Sigma_{a,\xi}\cap\Sigma_{a',\xi'}}=
P_{\Sigma_{a,\xi}}\circ P_{\Sigma_{a',\xi'}},
\end{multline}
because the normal vectors to $\Sigma_{a,\xi}$
and $\Sigma_{a',\xi'}$ are orthogonal. This property also holds 
for a vector of finite different scenarios $(\xi_1,\ldots,\xi_m)\in\Xi^m$ 
and arbitrary vector of arcs $(a_1,\ldots,a_m)\in\mathcal{A}^m$.
Based on this property, we propose four classes of algorithms 
in which, for every $i\in\{1,\ldots,m\}$, $S_i$ corresponds to the 
intersection of a
selection of $(\Sigma_{a,\xi})_{(a,\xi)\in 
\mathcal{A}\times\Xi}$ with different scenarios. This allows us to 
obtain explicit projections in our methods.

For defining the sets onto which we project,
for every $l\in\{1,\ldots,|\Xi|\}$, let
$\iota\colon \mathcal{A}^{l}\times D_l\to \{1,\ldots,m\}$ be a 
bijection, 
where $m=|\mathcal{A}^{l}\times D_l|$ and
\begin{align}
D_l=\Big\{(\xi_1,\ldots,\xi_l)\in\Xi^l \mid (\forall i,j\in 
\{1,\ldots,l\})\:\:i\ne 
j\Rightarrow \xi_i\ne\xi_j\Big\}.
\end{align} 
Observe that $D_1=\Xi$.
In order to obtain simple projection implementations
we define, for every $a=(a_i)_{i=1}^l\in\mathcal{A}^{l}$
and $ \xi=(\xi_i)_{i=1}^l\in 
D_{l}$,
\begin{equation}
K^{l}_{\iota(a,\xi)}=\bigcap_{i=1}^{l}\Sigma_{a_i,\xi_i}.
\end{equation}
Note that, for every $i\in\{1,\ldots,m\}$, $K_i^l$ is an intersection of a 
selection of $l$ sets from $(\Sigma_{a,\xi})_{a\in 
	\mathcal{A},\xi\in\Xi}$, where the considered scenarios are all 
different and the arcs are arbitrary. 
Since the selected scenarios are different, we obtain from 
\eqref{e:ortog} the explicit formula
\begin{equation}
(\forall a\in\mathcal{A}^l)(\forall \xi\in D_l)\quad 
P_{K^{l}_{\iota(a,\xi)}}=\prod_{i=1}^{l}P_{\Sigma_{a_i,\xi_i}}.
\end{equation}

For every $l\in\{1,\ldots,|\Xi|\}$, 
we propose four classes of algorithms based on 
Corollary~\ref{C:Corrollary 5.2},
depending on the selection 
of the convex sets $(K_i^l)_{1\le i\le m}$ in which they project, 
for different values of $m\in\{1,\ldots,|\mathcal{A}^l||D_l|\}$. 
\begin{enumerate}
\item[(F)] \textbf{Fixed selection: } We 
fix $j\in\{1,\ldots,|\mathcal{A}^l||D_l|\}$, we set $S=S_j=K_j^l$ and, 
for every iteration $k\in\NN$, we set
$\epsilon_{k}^{-1}(\{j\})=\Omega$. In this class, we project
deterministically in a fixed block of size $l$, in which arcs and 
scenarios are $\iota^{-1}(j)\in\mathcal{A}^l\times D_l$.

\item[(BA)] \textbf{Bernoulli alternating selection: }
We fix the bijection $\iota\colon \mathcal{A}^{l}\times D_l\to 
\{1,\ldots,|\mathcal{A}^l||D_l|\}$, fix $m\le |\mathcal{A}^l||D_l|$,
and, for every $i\in\{1,\ldots,m\}$, we set $S_i=K_i^l$. For every 
iteration 
$k\in\NN$, we set $I_k=\{0,(k\!\!\mod m)+1\}$.
In this class, 
we follow a $\{0,1\}-$Bernoulli random process, as illustrated in 
Remark~\ref{rem:1}\ref{rem:1Bernoulli}, to 
project onto blocks of size $l$ assigned cyclically along iterations.

\item[(DA)] \textbf{Deterministic alternating selection: }
We fix the bijection $\iota\colon \mathcal{A}^{l}\times D_l\to 
\{1,\ldots,|\mathcal{A}^l||D_l|\}$, fix $m\le |\mathcal{A}^l||D_l|$,
and, for every $i\in\{1,\ldots,m\}$, set $S_i=K_i^l$. For every iteration 
$k\in\NN$, we set 
$\epsilon_{k}^{-1}(\{0\})=\emp$.
In this class, 
we project deterministically onto blocks of size $l$ assigned cyclically 
along iterations.

\item[(RK)] \textbf{Random Kaczmarz selection: } We fix 
$m=|\mathcal{A}^{l}\times D_l|$, for every $i\in\{1,\ldots,m\}$, we set
$S_i=K_i^l$, for every  $k\in\NN$, we set $\epsilon_k$ as a 
$\{1,\ldots,m\}$-valued random variable 
 such that, for every $i\in\{1,\ldots,m\}$, 
$\pi^k_{i}=\frac{1}{m}$. In this class of algorithms, we randomly 
project onto a 
block of $l$ 
constraints from different scenarios at each iteration.
\end{enumerate}

\subsection{Numerical experiences}
\label{sec:numexp}
In this section, we apply the four classes of algorithms defined in 
Section~\ref{sec:select} to a specific instance of 
Problem~\ref{prob_transporte}.
We consider the network presented in \cite{Nguyen_dupuis}
(see also \cite{Xiaojun Chen}), represented 
in Figure~\ref{Imagen 1}.
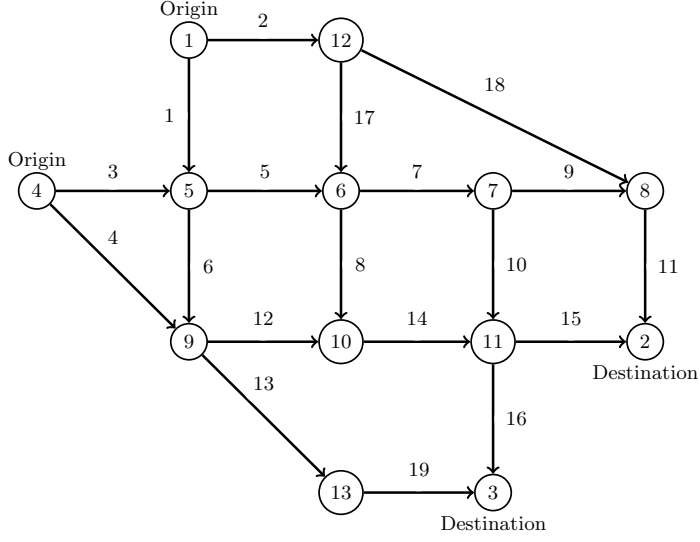
\begin{figure}
\centering
\scalebox{0.8}{\begin{tikzpicture}
\centering
\node (Origin) at (2.5,8) {Origin};
\node (Origin) at (0,5.5) {Origin};
\node (Destination) at (7.5,-0.5) {Destination};
\node (Destination) at (10,2) {Destination};
\begin{scope}[every node/.style={circle,thick,draw}]
\node (1) at (2.5,7.5) {1};
\node (2) at (10,2.5) {2};
\node (3) at (7.5,0) {3};
\node (4) at (0,5) {4};
\node (5) at (2.5,5) {5};
\node (6) at (5,5) {6} ;
\node (7) at (7.5,5) {7};
\node (8) at (10,5) {8};
\node (9) at (2.5,2.5) {9};
\node (10) at (5,2.5) {10};
\node (11) at (7.5,2.5) {11};
\node (12) at (5,7.5) {12} ;
\node (13) at (5,0) {13} ;
\end{scope}
\begin{scope}[
  every node/.style={fill=white,circle},
  every edge/.style={draw=black,very thick}]
  \path [->] (4) edge node[above] {$3$} (5);
  \path [->] (6) edge node[above] {$7$} (7);
  \path [->] (7) edge node[above] {$9$} (8);
  \path [->] (9) edge node[above] {$12$} (10);
  \path [->] (10) edge node[above] {$14$} (11);
  \path [->] (12) edge node[above =0.15 cm] {$18$} (8);
  \path [->] (4) edge node[above =0.15 cm] {$4$} (9);
   \path [->] (1) edge node[left] {$1$} (5);
  \path [->] (12) edge node[right] {$17$} (6);
  \path [->] (6) edge node[right] {$8$} (10);
  \path [->] (7) edge node[right] {$10$} (11);
  \path [->] (8) edge node[right] {$11$} (2);
\end{scope}
\begin{scope}[
  every node/.style={fill=white,circle},
  every edge/.style={draw=black,very thick}]
  \path [->] (1) edge node[above] {$2$} (12);
  \path [->] (5) edge node[above] {$5$} (6);
  \path [->] (11) edge node[above] {$15$} (2);
  \path [->] (13) edge node[above] {$19$} (3);
  \path [->] (9) edge node[above =0.15 cm] {$13$} (13);
  \path [->] (5) edge node[right] {$6$} (9);
  \path [->] (11) edge node[right] {$16$} (3);
\end{scope}
\end{tikzpicture}}
\caption{Network with $|\mathcal{A}|=19$, 
$\mathcal{O}=\{1,4\}$, $\mathcal{D}=\{2,3\}$
, $|R_{1,2}|=8$, 
$|R_{4,3}|=|R_{1,3}|=6$, $|R_{4,2}|=5$, and $|R|=25$
\cite{Nguyen_dupuis}.}
\label{Imagen 1}
\end{figure}
We set $|\Xi|=18$, 
$p_{\xi}\equiv\frac{1}{|\Xi|}$, and 
$(c_{\xi})_{\xi\in\Xi}$ as a sample of the random variable 
$c+\kappa\cdot \mbox{Beta}(20,20)$, where
$c$ and $\kappa$ are in Table \ref{Numerical values}.
The demand 
$(h_{\xi})_{\xi\in\Xi}$ is 
obtained as 
a sample of the random variable $d+s\cdot \text{Beta}(50,10)$, 
where $d=(d_{1,2},d_{1,3},d_{4,2},d_{4,3})=(300,700,500,350)$ and 
$s=(120,120,120,120)$ and we consider the capacity expansion 
limits, for every $a\in\mathcal{A}$, $M_a=200\cdot\kappa_a$. 
The matrix of the quadratic cost of expansion is given by 
$Q=\Id_{|\mathcal{A}|}$ and we consider the travel time function 
\begin{align}
(\forall \xi\in\Xi)(\forall 
a\in\mathcal{A})\hspace{0.3cm}t_{a,\xi}(u):=\eta_{a}+\tau_{a}\frac{u}{c_{a,\xi}},\label{ecuacion
52}
\end{align}
where $\eta$ is in Table \ref{Numerical values} and $\tau:=0.15\eta$. 
Hence, for every $\xi\in\Xi$ and 
$a\in\mathcal{A}$, $\beta_{a,\xi}:=\frac{\tau_{a}}{c_{a,\xi}}$ .
\begin{table}[]
\centering
\begin{tabular}{c c}
\begin{tabular}{|c|c|c|c|}
\hline
Arcs & $c$  & $\kappa$ & $\eta$ \\ \hline
1    & 1100 & 15       & 7      \\ \hline
2    & 484  & 6.6      & 9      \\ \hline
3    & 154  & 2.1      & 9      \\ \hline
4    & 1100 & 15       & 12     \\ \hline
5    & 330  & 4.5      & 3      \\ \hline
6    & 484  & 6.6      & 9      \\ \hline
7    & 1100 & 15       & 5      \\ \hline
8    & 220  & 3        & 13     \\ \hline
9    & 220  & 3        & 5      \\ \hline
10   & 220  & 6        & 9      \\ \hline
\end{tabular} &  
\begin{tabular}{|c|c|c|c|}
\hline
Arcs & $c$ & $\kappa$ & $\eta$ \\ \hline
11   & 770 & 10.5     & 9      \\ \hline
12   & 770 & 10.5     & 10     \\ \hline
13   & 770 & 10.5     & 9      \\ \hline
14   & 770 & 10.5     & 6      \\ \hline
15   & 440 & 6        & 9      \\ \hline
16   & 385 & 5.25     & 8      \\ \hline
17   & 242 & 3.3      & 7      \\ \hline
18   & 220 & 6.6      & 14     \\ \hline
19   & 440 & 10.5     & 11     \\ \hline
\end{tabular}
\end{tabular}
\caption{
Numerical values of $c$, $\kappa$, and $\eta$ on every arc.}
\label{Numerical values}
\end{table}

In this context, we apply the four classes of algorithms defined in 
Section~\ref{sec:select} with $l\in\{1,9,18\}$, which give raise to 
12 algorithms. We compare their numerical performance with respect 
to the method without any projection proposed in 
 \cite{Condat,Vu}, called Algorithm 1. This comparison is performed 
by creating $20$ random instances of the problem,  obtained via 
the random function of MATLAB and using the same seed.  
In Table~\ref{table:Algorithm} 
 we detail the 
algorithm labelling according to the class and the number of 
constraints onto which we project ($l\in\{1,9,18\}$).
For the class of fixed selections, we choose to project onto 
the polyhedron related to capacity constraints of arc $a=16$, i.e., 
$\bigcap\limits^{l}_{j=1}\Sigma_{16,\xi_j}$,
where $(\xi_1,\ldots,\xi_l)\in D_l$. This is 
naturally justified from 
the topology of the network in Figure~\ref{Imagen 1}, since
the total demand arriving to node 3 exceeds largely the capacity 
in arcs $16$ and $19$. Thus, it is mandatory to expand the capacity 
of those arcs. 
Moreover, the arc $16$ has the lowest capacity.

 \begin{table}[]
\begin{tabular}{|l|c|c|c|}
\hline
Class           & $l=1$ & $l=9$ & $l=18$ \\ \hline
(F) Fixed selections                 & Alg. 2     & Alg. 3     & Alg. 4      \\ 
\hline
(BA) Bernoulli alternating selection & Alg. 5     & Alg. 6   & Alg. 7      \\ 
\hline
(DA) Deterministic alternating selection  & Alg. 8     & Alg. 9     & Alg. 
10     
\\ \hline
(RK) Randomized Kaczmarz selection & Alg. 11  & Alg. 12   & Alg. 
13     \\ 
\hline
\multicolumn{4}{|c|}{No projections: Alg.1}                 \\ \hline
\end{tabular}
\caption{
Algorithm labelling according to the class and $l$.}
\label{table:Algorithm}
\end{table}
All algorithms stop at the first iteration when the relative error
is less than a tolerance of $10^{-10}$, where the relative error of the 
iteration $
k\in\NN$ is
\begin{equation}
e_k=\sqrt{\frac{\|x^{k+1}-x^{k}\|^2+\|f^{k+1}-f^{k}\|^2+\|u^{k+1}-u^{k}\|^2+\|v^{k+1}-v^{k}\|^2}{\|x^{k}\|^2+\|f^{k}\|^2+\|u^{k}\|^2+\|v^{k}\|^2}}.
\end{equation}
\begin{table}[]
\centering
\scalebox{0.8}{
\begin{tabular}{c c c}
& \begin{tabular}{|l|c|c|}
\hline Alg                                                                                                 
& Time {[}s{]} & Iter \\ \hline
Alg. 1                                                                                     & 
110.10       & 34206                   \\ \hline

\end{tabular}&\\

$l=1$ & $l=9$ & $l=18$\\
\begin{tabular}{|l|c|c|}
\hline Alg                                                                                                
& Time {[}s{]} & Iter \\ \hline
Alg. 2                   & 108.60       & 34273  \\ \hline
Alg. 5 & 106.91       & 33784             \\ \hline
Alg. 8         & 105.66       & 33342            \\ \hline
Alg. 11       & 105.74      & 33381            \\ \hline
\end{tabular}
&
\begin{tabular}{|l|c|c|}
\hline Alg                                                                                                
& Time {[}s{]} & Iter \\ \hline

Alg. 3                   & 110.31       & 34206               \\ \hline
Alg. 6  & 98.07       & 30850               \\ \hline
Alg. 9           & 90.60       & 28193             \\ \hline
Alg. 12           & 92.57       & 28664             \\ \hline
\end{tabular}
&  \begin{tabular}{|l|c|c|}
\hline Alg                                                                                                 
& Time {[}s{]} & Iter \\ \hline
Alg. 4                   & 112.17       & 34062  \\ \hline
Alg. 7 & 90.05       & 28071\\ \hline
Alg. 10           & 79.86     & 24472                          \\ \hline
Alg. 13         & 83.46        & 25085            \\ \hline
\end{tabular}

\end{tabular}}
\caption{
The average execution time and the average number of iterations of 
each method.}
\label{table1:R25A19E18}
\end{table}
In Table~\ref{table1:R25A19E18} we provide
the average
execution time and the average number of iterations of Algorithms 
$1-13$,  
obtained from 20 random realizations of vectors $(c_\xi)_{\xi\in\Xi}$ 
and 
$(d_\xi)_{\xi\in\Xi}$. 
We see a considerable
decrease on the average time and number of 
iterations of the algorithms in classes
(BA), (DA), and (RK),
as we increase the number of constraints $l$ considered in the 
projections. The algorithms in class (F) remains comparable 
with respect to Algorithm 1. 

In Figure~\ref{Imagen 2} we provide 
the boxplot of \% of improvement in terms of number of iterations of 
Algorithms $2-13$ with respect
to Algorithm~1 (without projections). A similar boxplot in terms of 
computational time is provided in Figure~\ref{Imagen 3}. We verify 
that Algorithms 
2, 3 \& 4 belonging to the (F) class, are comparable in efficiency to
Algorithm~1. We see that all other algorithms have a superior 
performance at the exception of few outliers. In particular, Algorithms
10 (DA) and 13 (RK) exhibit larger gains in performance, reaching up 
to 35\% of improvement in iterations and up to 38\% in computational 
time. We also observe that the algorithms in which we project onto a 
larger number of constraints (larger $l$) have better performance.

\begin{figure}[h]
\centering
\includegraphics[width=1\linewidth]{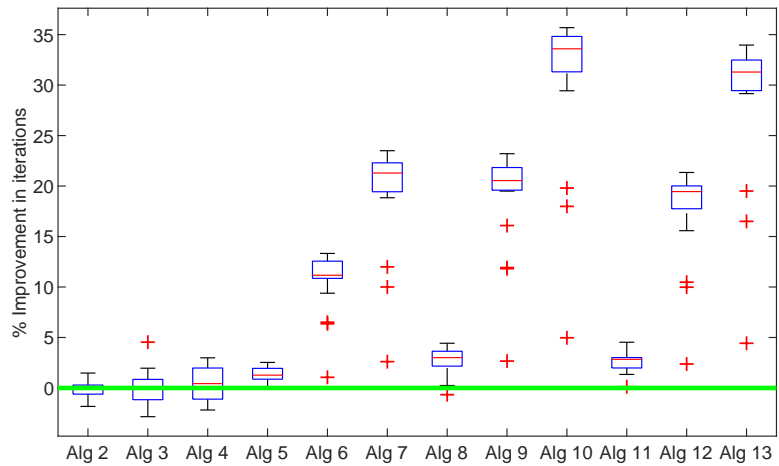}
\caption{Boxplot of \% of improvement in number of iterations for 
Algorithms 2-13 with respect to Algorithm 1.}
\label{Imagen 2}
\end{figure}

\begin{figure}[h]
\centering
\includegraphics[width=1\linewidth]{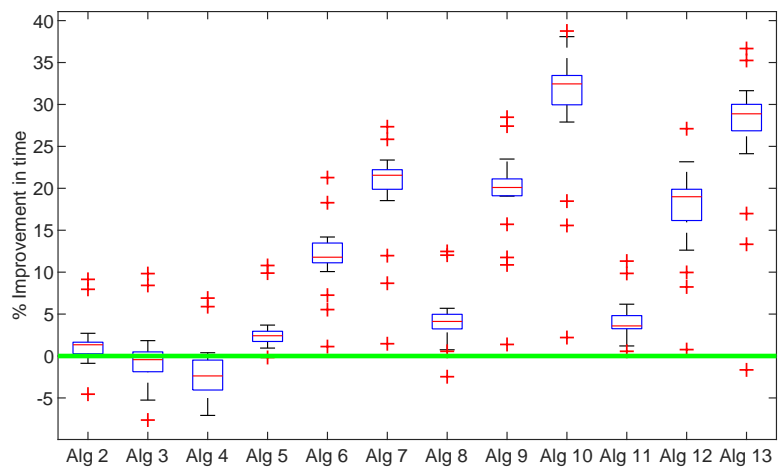}
\caption{Boxplot of \% of improvement in time for 
Algorithms 2-13 with respect to Algorithm 1.}
\label{Imagen 3}
\end{figure}

\begin{table}[]
\centering
\begin{tabular}{c c}
\begin{tabular}{|l|c|c|}
\hline
Arco & $\max_{\xi} (u_{a,\xi}-c_{a,\xi})$ & $x_{a}$  \\ \hline
1    & -394.91                              & 0.00  \\ \hline
2    & 18.20                                & 18.20 \\ \hline
3    & -6.95                                & 0.00  \\ \hline
4    & -187.97                              & 0.00  \\ \hline
5    & 24.92                                & 24.92 \\ \hline
6    & 15.62                                & 15.62 \\ \hline
7    & -633.86                              & 0.00  \\ \hline
8    & -221.02                              & 0.00  \\ \hline
9    & -73.31                               & 0.00  \\ \hline
10   & -95.52                               & 0.00  \\ \hline
\end{tabular}
&\begin{tabular}{|l|c|c|}
\hline
Arco & $\max_{\xi} (u_{a,\xi}-c_{a,\xi})$ & $x_{a}$  \\ \hline
11   & -218.34                              & 0.00  \\ \hline
12   & -164.31                              & 0.00  \\ \hline
13   & 36.20                                & 36.20 \\ \hline
14   & -164.31                              & 0.00  \\ \hline
15   & 17.67                                & 17.67 \\ \hline
16   & 68.42                                & 68.42 \\ \hline
17   & -126.34                              & 0.00  \\ \hline
18   & -78.68                               & 0.00  \\ \hline
19   & 36.20                                & 36.20 \\ \hline
\end{tabular}
\end{tabular}
\caption{Worst scenario 
flow excess  on arcs $(\max_{\xi} (u_ 
{a,\xi}-c_{a,\xi}))_{a\in\mathcal{A}}$ and
arc capacity expansion vector $(x_a)_{a\in\mathcal{A}}$
at the optimum.}
\label{T: Flows}
\end{table}
\begin{figure}
\centering
\scalebox{0.8}{\begin{tikzpicture}
\centering
\node (Origin) at (2.5,8) {Origin};
\node (Origin) at (0,5.5) {Origin};
\node (Destination) at (7.5,-0.5) {Destination};
\node (Destination) at (10,2) {Destination};
\begin{scope}[every node/.style={circle,thick,draw}]
\node (1) at (2.5,7.5) {1};
\node (2) at (10,2.5) {2};
\node (3) at (7.5,0) {3};
\node (4) at (0,5) {4};
\node (5) at (2.5,5) {5};
\node (6) at (5,5) {6} ;
\node (7) at (7.5,5) {7};
\node (8) at (10,5) {8};
\node (9) at (2.5,2.5) {9};
\node (10) at (5,2.5) {10};
\node (11) at (7.5,2.5) {11};
\node (12) at (5,7.5) {12} ;
\node (13) at (5,0) {13} ;
\end{scope}
\begin{scope}[
  every node/.style={fill=white,circle},
  every edge/.style={draw=black,very thick}]
  \path [->] (4) edge node[above] {$3$} (5);
  \path [->] (6) edge node[above] {$7$} (7);
  \path [->] (7) edge node[above] {$9$} (8);
  \path [->] (9) edge node[above] {$12$} (10);
  \path [->] (10) edge node[above] {$14$} (11);
  \path [->] (12) edge node[above =0.15 cm] {$18$} (8);
  \path [->] (4) edge node[above =0.15 cm] {$4$} (9);
   \path [->] (1) edge node[left] {$1$} (5);
  \path [->] (12) edge node[right] {$17$} (6);
  \path [->] (6) edge node[right] {$8$} (10);
  \path [->] (7) edge node[right] {$10$} (11);
  \path [->] (8) edge node[right] {$11$} (2);
\end{scope}
\begin{scope}[
  every node/.style={fill=white,circle},
  every edge/.style={draw=red,very thick}]
  \path [->] (1) edge node[above] {$2$} (12);
  \path [->] (5) edge node[above] {$5$} (6);
  \path [->] (11) edge node[above] {$15$} (2);
  \path [->] (13) edge node[above] {$19$} (3);
  \path [->] (9) edge node[above =0.15 cm] {$13$} (13);
  \path [->] (5) edge node[right] {$6$} (9);
  \path [->] (11) edge node[right] {$16$} (3);
\end{scope}
\end{tikzpicture}}
\caption{Graphical representation of the expanded arcs (in red, 
$x_a>0$) at the optimum.}
\label{G: Graph1}
\end{figure}
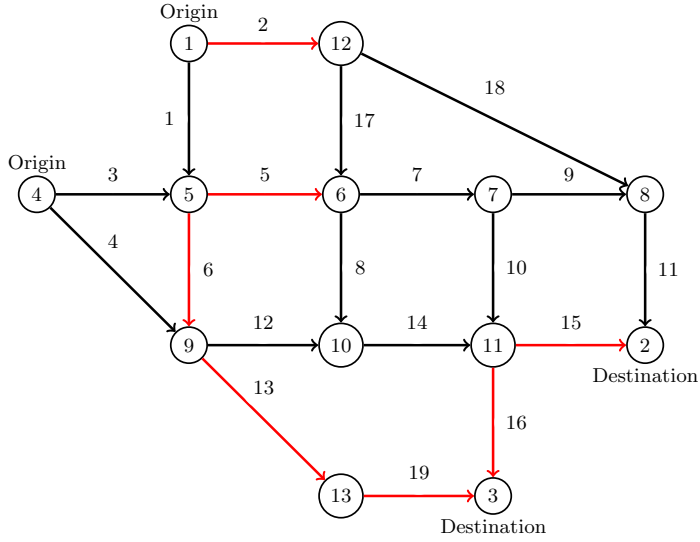

In terms of the obtained solution, 7 arcs are 
expanded and the expansion capacity coincides with the extra flow 
needed in the equilibrium for the worst scenario.
Finally, in Table~\ref{T: Flows} we show, at the optimum, the flow 
excess at each arc
in the worst scenario and the corresponding arc capacity expansion
for one of the 20 random realizations. 
We verify that the arc capacity expansion coincides with the worst 
scenario flow excess on arcs where the excess is strictly positive.
In the arcs in which there is a slack on the capacity, the expansion 
is zero. In Figure~\ref{G: Graph1} we represent the expanded arcs.

\section{Conclusions}
In this work, we provide a new primal-dual algorithm for solving 
monotone inclusions with a priori information. The a priori 
information  is represented via fixed point sets of a finite number of 
nonexpansive operators, which are activated 
randomly/deterministically in our proposed method.
We apply four classes of algorithms 
with different activation schemes for solving convex optimization with 
a priori information
and, in particular, to the arc capacity expansion problem on 
traffic networks. We observe an improvement up to 35\% in 
computational time
for the algorithms including randomized and alternating 
projections with respect to the method without projections, 
justifying the advantage of our approach.

\begin{acknowledgements}
The work of the first and second authors are founded by the National 
Agency for Research and Development (ANID) under grants FONDECYT 
1190871 and FONDECYT 11190549, respectively. The third autor 
thanks to the ``{Direcci\'on de Postgrados y Programas de la 
Universidad T\'ecnica Federico Santa Mar\'ia}''.
\end{acknowledgements}


\end{document}